\documentclass[10pt]{article}
\usepackage{amsfonts}
\usepackage{amsfonts,latexsym}
\usepackage{amsmath}
\usepackage{amsthm}
\usepackage{amssymb,bbm}
\usepackage{mathrsfs}
\usepackage{color}
\usepackage{amsmath}
\usepackage{tikz}  
\usetikzlibrary{arrows,shapes,chains}
\usepackage{authblk}
\usepackage{indentfirst}
\usepackage{setspace}
\usepackage[colorlinks,linkcolor=blue,citecolor=blue]{hyperref}
\usepackage{geometry}
\allowdisplaybreaks
\newtheorem{theorem}{\indent Theorem}[section]
\newtheorem{proposition}[theorem]{\indent Proposition}

\newtheorem{lemma}[theorem]{\indent Lemma}

\begin{document}	
	\title{\bf Ergodicity for the fractional Magneto-Hydrodynamic equations driven by a degenerate pure jump noise\footnote{The research is  supported by National Natural Science Foundation of China (No.12371198, 12031020), Scientific Research Program Foundation of National University of Defense Technology (No.22-ZZCX-016) and Hunan Provincial Natural Science Foundation of China (No.2024JJ5400,2025JJ50037)}}
	
	\author[a]{Xue Wang} 
\author[a]{Jianhua Huang \thanks{Corresponding author:\,\,\; jhhuang32@nudt.edu.cn} }
    \author[b]{Jiangwei Zhang}
    
		\affil[a]{\small  \it College of Science, National University of Defense Technology, Changsha, 410073, P.R.China}	
\affil[b]{\small  \it Institute of Applied Physics and Computational Mathematics, Beijing 100088, P.R.China}

	\date{}
		
\maketitle

\begin{abstract}
This paper is concerned with the ergodicity for stochastic 2D fractional magneto-hydrodynamic equations on the two-dimensional torus driven by a highly degenerate pure jump L\'{e}vy noise. We focus on the challenging case where the noise acts in as few as four directions, establishing new results for such high degeneracy. 
We first employ Malliavin calculus and anticipative stochastic calculus to demonstrate the equi-continuity of the semigroup (or so-called the e-property), and then verify the weak irreducibility of the solution process. Therefore, the  uniqueness of invariant measure is proven. 




\medskip

\end{abstract}

{\hspace*{2mm} AMS Subject Classification: 60H15, 60G51, 76D06, 76M35.}

\linespread{1.2}
\section{Introduction}
\numberwithin{equation}{section}
The magneto-hydrodynamic (MHD) equations mathematically characterize the evolution of velocity and magnetic fields in conductive fluids, while embedding fundamental physical conservation laws \cite{Cababbes,Cowling}. The deterministic MHD equations have been extensively studied, and the existence and stability of solutions to such equations provide key insights into the stability and turbulence of magnetohydrodynamic systems (see e.g., \cite{Foias-2001,Sermange-1983}). However, in many real-world applications, random disturbances caused by external forces, the presence of thermal noise, or uncertainties in initial conditions  cannot be ignored. Therefore, the study of stochastic MHD equations has become extremely important. In fact, research on stochastic MHD equations not only demonstrates the applicability of their theory to more realistic scenarios but also presents significant mathematical challenges, particularly regarding the impact of noise injection on the long-term behavior, ergodicity, and stability of the solutions to these equations (see \cite{Glatt-2008,MH-2006} and the references therein).

The ergodic property of MHD systems with stochastic perturbations have been systematically investigated under various types of noise. For non-degenerate stochastic forcing acting jointly on velocity and magnetic fields, the coupling method introduced in \cite{Barbu-2007} established the existence and uniqueness of invariant measures. Afterwards, Huang and Shen \cite{Huang-2016} addressed the 2D incompressible fractional stochastic MHD equations perturbed by nonlinear multiplicative Gaussian noise, where they rigorously analyzed global well-posedness and proved the existence of a compact random attractor. Hong et al. \cite{Hong-2021} considered the well-posedness and exponential mixing for stochastic MHD equations driven by multiplicative noise.
Recently, Luo \cite{Luo-2022} studied the problem of regularization by multiplicative noise of transport type for the 3D MHD equations. 
However, the aforementioned noises are non-degenerate, acting uniformly across all Fourier modes. A natural question arises: How would the statistical properties of the system be altered if the noise excitation were constrained to only a few directions?
In fact, some results have been obtained in this regard, although very few.
Specifically, Peng et al. \cite{Peng-2020} established the existence, uniqueness and exponential attraction properties of an invariant measure for the MHD equations with additive degenerate stochastic forcing acting only in the magnetic equation. 
It is worth mentioning that they \cite{Peng-2024} established the ergodicity for stochastic 2D Navier-Stokes equations driven by a highly degenerate pure jump Lévy noise in recent. Following this result,
Huang et al. \cite{Huang-2025} have extended this line of research to coupled systems, proving ergodicity for the Boussinesq equations under degenerate pure jump stochastic forcing that acts only in the temperature equation. However, the ergodic theory for MHD systems driven by degenerate pure jump noise remains unexplored. This constitutes a significant gap in the literature, particularly given the physical relevance of such models in plasma turbulence where intermittent, jump-like perturbations naturally occur. While the existing techniques for degenerate Gaussian noise \cite{Foldes-2015,MH-2006,MH-2008,MH-2011} and recent progress on pure jump cases \cite{Huang-2025,Peng-2024} provide valuable tools, the coupled velocity-magnetic field structure of MHD equations introduces new mathematical challenges.
In this paper, we focus on the ergodicity of the following stochastic MHD equations with a degenerate pure jump noise on a 2D torus $\mathbb{T}^2=[-\pi,\pi]^2=\mathbb{R}^2/(2\pi\mathbb{Z}^2)$
\begin{eqnarray}\label{MHD1.1}
\left\{
\begin{aligned}
&\mathrm{d}u+[u\cdot\nabla u+\nu_1(-\Delta)^{\alpha}u]\mathrm{d}t=[-\nabla p+b\cdot \nabla b]\mathrm{d}t,\\
&\mathrm{d}b+[u\cdot\nabla b+\nu_2(-\Delta)^{\beta}b]\mathrm{d}t=b\cdot\nabla u \mathrm{d}t
+\sum_{k=(k_1,k_2)\in \mathcal Z_0}\left(\frac{k_2}{|k|},-\frac{k_1}{|k|}\right)^T \alpha_k^0 \cos(k\cdot x)\mathrm{d}W_{S_t}^{k,0}\\
&~~~~~~~~~~~~~~~~~~~~~~~~~~~~~~~~~~~~~+\sum_{k=(k_1,k_2)\in \mathcal Z_0}\left(-\frac{k_2}{|k|},\frac{k_1}{|k|}\right)^T \alpha_k^1 \sin(k\cdot x)\mathrm{d}W_{S_t}^{k,1},\\
&\nabla\cdot u=\nabla\cdot b=0,\\
&u(0,x)=u_0(x),~b(0,x)=b_0(x)
\end{aligned}
\right.
\end{eqnarray}
with periodic boundary conditions 
\begin{equation*}
u_i(x,t)=u_i(x+2\pi j,t),~b_i(x,t)=b_i(x+2\pi j,t),~i=1,2,
\end{equation*}
where $u=(u_1,u_2)$ and $b=(b_1,b_2)$ denote the velocity field and magnetic field respectively, $p$ is a scalar pressure, $\nu_1$ is the kinematic viscosity, and $\nu_2$ is the magnetic diffusion.
$\mathcal Z_0$ is a subset of $\mathbb{Z}^2 \setminus\{0,0\}$, $W_{S_t}=(W_{S_t}^{k,m})_{k\in \mathcal Z_0,m\in\{0,1\}}$ is a $2|\mathcal Z_0|$-dimensional subordinated Brownian motion defined relative to a filtered probability space $(\Omega,\mathcal{F},\{\mathcal{F}_t\}_{t\geq0},\mathbb{P})$, $\{\alpha_k^m\}_{k\in \mathcal Z_0,m\in\{0,1\}}$ are non-zero constants.


Recent advances in ergodic theory for infinite-dimensional dynamical systems subjected to degenerate stochastic perturbations have generated substantial theoretical interest (\cite{EM01,Foldes-2015,MH-2006,MH-2008,MH-2011}).  
This surge stems both from intrinsic mathematical complexity, challenging classical finite-dimensional paradigms, and the imperative to rigorously validate statistical hypotheses underpinning physical modeling frameworks. Notably, pioneering works (e.g., \cite{Foldes-2015,MH-2008,MH-2011}) have laid foundational stones for hypoelliptic analysis in stochastic partial differential equations (SPDEs) with degenerate forcing, marking the emergence of a unified framework. Nevertheless, the theory remains in its nascent phase, constrained by the intricate interplay between infinite-dimensional geometry and stochastic analysis. A central obstruction lies in establishing the Malliavin covariance matrix's non-degeneracy, a task complicated by the absence of ellipticity and the spectral complexity inherent to SPDEs, rendering even spectral gap characterization intractable. To circumvent these barriers, researchers have leveraged structural properties of turbulent dynamical systems: Despite their infinite-dimensional nature, such systems exhibit finite-dimensional unstable manifolds in phase space. Consequently, current strategies prioritize demonstrating that restricted Malliavin matrices attain uniform positivity on carefully constructed conic subsets spanning these unstable directions.
However, as emphasized in the articles \cite{Foldes-2015,MH-2006,MH-2011,Peng-2020}, the technical difficulties of this method lie in how to generate successively larger finite dimensional spaces through the interaction between the nonlinear and stochastic terms and how to exert delicate spectral analysis on these spaces.

The key idea relies on computing Lie brackets from constant vector fields and nonlinear advective terms through an iterative process. This approach systematically breaks down the phase space of the system. The decomposition helps find minimal spanning conditions that satisfy Hörmander-type hypoellipticity requirements. Importantly, the choice of Lie bracket sequences depends on the specific model, there's no general method that works for all cases. For instance, Navier-Stokes systems and Boussinesq flows need different analytical treatments (see \cite{Foldes-2015,MH-2008,MH-2011}), as each requires unique modifications of the Hörmander framework to handle their particular nonlinear structures. 
The construction of a suitable Hörmander-type framework for the MHD system presents distinct challenges, particularly in handling the coupled velocity-magnetic field structure. However, building upon the foundational work of Peng et al. \cite{Peng-2020} for degenerate stochastic MHD equations and their subsequent extension to pure jump processes \cite{Peng-2024,Peng-2025} for Navier-Stokes systems, we develop a more streamlined approach for the spectral analysis of MHD systems driven by degenerate pure jump noise. The principal theoretical advances of this work are twofold:
%
%

$(i)$ Combining the Lie bracket algorithm tailored to the fractional MHD system in \cite{Peng-2020} with the generalized conditions extended to pure jump noise (see Lemma \ref{MHDlemma3.2}), enabling the derivation of a non-standard H\"{o}rmander criterion;

$(ii)$ Without performing vorticity transformation, we derive the ergodicity of fractional MHD equation under the influence of degenerate pure jump L\'evy noise and establish precise spectral estimates.


The paper is organized as follows. In Section 2, we give functional setting, notations and the Markovian framework, as well as the control problem, and main results, i.e. weak irreducibility for the solution process of the system \eqref{MHD1.1} (see Proposition \ref{MHDirre}), the equi-continuity of the semigroup, the so-called e-property (see Proposition \ref{MHDpropo1.4}), and the existence and uniqueness of the invariant measure for the system \eqref{MHD1.1} (see Theorem \ref{MHDergodicity}). In Section 3, we give some moment estimates of the solution 
for the stochastic fractional MHD system with pure jump L\'{e}vy noise. Section 4 devotes to the invertibility of Malliavin matrix and a dissipative property of fractional MHD system. We will state proofs of the main theorem in Section 5. Some technical proofs are provided in the Appendix for completeness.

\section{Preliminaries}
This section is organized as follows: First, we establish the fundamental notation that will be used throughout this work. Second, we formulate the control problem associated with the stochastic fractional MHD system \eqref{MHD1.1}. Furthermore, we introduce the mathematical framework for subordinated L\'{e}vy noise. Finally, we present the main theoretical results of this study.
\subsection{Notations}
This subsection mainly introduce some notations and functional setting.
Denote the phase space as the collection of square integrable functions with zero mean
\begin{equation*}
H_1^s:=\left\{u:=(u_1,u_2)\in(W^{s,2}(\mathbb{T}^2))^2:
\nabla\cdot u=0,
\int_{\mathbb{T}^2}u_1(x)\mathrm{d}x
=\int_{\mathbb{T}^2}u_2(x)\mathrm{d}x=0\right\}
~\text{for any}~s\geq 0,
\end{equation*}
where $W^{s,2}(\mathbb{T}^2)$ is classical Sobolev space, and $H_1^s$
is equipped with the norm
\begin{equation*}
\|u\|_{H_1^s}:=(\|u_1\|_{W^{s,2}(\mathbb{T}^2)}^2 + \|u_2\|_{W^{s,2}(\mathbb{T}^2)}^2)^{\frac{1}{2}}, ~\forall\, u=(u_1,u_2).
\end{equation*}
Assume that $H_2^s$ has the same properties as $H_1^s$, and let $H^s = H_1^s \times H_2^s $. The norm on the space $H^s$ is defined by
\begin{equation*}
\|U\|_{s}:=\|U\|_{H^s}=(\|u\|_{H_1^s}^2+\|b\|_{H_2^s}^2)^{\frac{1}{2}},~\forall\, U = (u,b) \in H^s.
\end{equation*}
We denote by $(H_1^{s})^*$, $(H_2^{s})^*$, and $(H^s)^*$ the dual spaces to $(H_1^{s})$, $(H_2^{s})$ and $H^{s}$, respectively. Specially,
\begin{equation*}
H_1^0=\left\{u:=(u_1,u_2)\in(L^2(\mathbb{T}^2))^2:
\nabla\cdot u=0,
\int_{\mathbb{T}^2}u_1(x)\mathrm{d}x
=\int_{\mathbb{T}^2}u_2(x)\mathrm{d}x=0\right\}.
\end{equation*}
The norm on the space $H_1^0$ is given by
\begin{equation*}
\|u\|=\|(u_1,u_2)\|:=(\|u_1\|_{L^2(\mathbb T^2)}^2 + \|u_2\|_{L^2(\mathbb T^2)}^2)^{\frac{1}{2}}.
\end{equation*}
Equally, let $H:=H^0=H_1^0\times H_2^0$. With a minor notational overlap, $\langle \cdot, \cdot \rangle$ can represent the inner product on either the Hilbert space $H$ or $H_1^0,H_2^0$. Let $\Pi$ denote the projection operator from $(L^2(\mathbb{T}^2))^2$ to the space $H_1^0$ or $H_2^0$.

For the sake of convenience and simplicity, let $\Lambda^\alpha u = (-\Delta)^{\alpha/2} u$ for any $ u \in H_1^\alpha $, and let $\Lambda^\beta b = (-\Delta)^{\beta/2} b$ for any $b \in H_2^\beta$.

Fix the trigonometric basis:
\begin{equation}
\sigma_{k}^{0}(x):=\left(0,e_k^0\right)^{T}\in H_1^0\times H_2^0,
\quad\sigma_{k}^{1}(x):=\left(0,e_k^1\right)^{T}\in H_1^0\times H_2^0,
\end{equation}
and
\begin{equation}\label{MHDpsi}
\psi_{k}^{0}(x):=\left(e_k^0,0\right)^{T}\in H_1^0\times H_2^0,
\quad\psi_{k}^{1}(x):=\left( e_k^1,0 \right)^{T}\in H_1^0\times H_2^0,
\end{equation}
where $k=(k_1, k_2)\in\mathbb{Z}^2/ \{(0,0)\}$, and 
\begin{equation*}
e_k^0 = \left(\frac{k_2}{|k|}, -\frac{k_1}{|k|} \right)^T \cdot \cos(k \cdot x),
~~
e_k^1 = \left(-\frac{k_2}{|k|}, \frac{k_1}{|k|} \right)^T \cdot \sin(k \cdot x),
\end{equation*}
are orthogonal basis of $H$, this allows us to construct stochastic forcing.

Let
$\{e_k^l\}_{k\in \mathcal{Z}_0,l=0,1}$ be the standard basis of $\mathbb{R}^{2|\mathcal{Z}_0|}$ and $\{\alpha_k^l\}_{k\in \mathcal{Z}_0,l\in\{0,1\}}$ be a sequence of non-zero numbers. We define a linear map  $Q_{b}:\mathbb{R}^{2|\mathcal{Z}_0|}\rightarrow H$ such that
\begin{equation*}
Q_{b}e_k^l:=\alpha_k^l\sigma_k^l.
\end{equation*}
Denote the Hilbert-Schmidt norm of $Q_b$ by 
$$
\mathcal{B}_0:=\|Q_{b}\|^2=\|Q_{b}^* Q_{b}\|
=\sum_{k\in \mathcal{Z}_0,l\in\{0,1\}}(\alpha_k^l)^2.
$$
Then the stochastic forcing can be expressed in the following form
\begin{equation}\label{MHDsigma}
Q_{b}\mathrm{d}W_{S_t}:=\sum_{k\in \mathcal{Z}_0,l\in\{0,1\}}\alpha_k^l\sigma_k^l\mathrm{d}W_{S_t}^{k,l}.
\end{equation}


For $U=(u,b)$ and $\tilde{U}=(\tilde{u},\tilde{b})$, denote $A^{\alpha,\beta}U=(\nu_1(-\Delta)^{\alpha}u,\nu_2(-\Delta)^{\beta}b)^T$, and 
\begin{equation*}\begin{aligned}
B(U,\tilde{U}) & =\quad
\begin{pmatrix}
\Pi[u\cdot\nabla\tilde{u}-b\cdot\nabla\tilde{b}] \\
\Pi[u\cdot\nabla\tilde{b}-b\cdot\nabla\tilde{u}]
\end{pmatrix}, \\
B(U) & :=\quad B(U,U)\\
F(U)&:=-A^{\alpha,\beta}U-B(U).
\end{aligned}\end{equation*}
The bilinear operator $\tilde{B}(\cdot,\cdot)$ can be defined by
\begin{equation*}
\langle \tilde{B}(u,v),w \rangle =  \int\limits_{\mathbb{T}^2} \Pi(u \cdot \nabla v) w \, \mathrm{d}x, \quad u,v,w \in H_k^0 \cap H_k^1,~k=1,2.
\end{equation*}

With the above notations, the equations \eqref{MHD1.1} can be rewritten as an
abstract stochastic evolution equation on $H$:
\begin{equation}\label{MHD2.14}
\mathrm{d}U+(A^{\alpha,\beta}U+B(U,U))\mathrm{d}t=Q_b\mathrm{d}W_{S_t},~~U(0)=U_0=(u_0,b_0).
\end{equation}
The proof of the well-posedness of system \eqref{MHD2.14} is similar to that of \cite{Huang-2016}. Hence,
we can denote by $U = U(t, U_0)$ the unique solution of system \eqref{MHD2.14} with initial value $U_0$.

\subsection{Deriving the control problem}
Denote by $M_b(H)$ and $C_b(H)$ the spaces of bounded measurable and bounded continuous real-valued functions on $H$, respectively.

Fix $U=U(t,U_0)=U(t,U_0,W_{S_t})$, and note that by well-posedness of correspongding system the solution defines a Markov process. Accordingly, for every Borel set $A\subseteq H$, the transition function associated to \eqref{MHD2.14} is given by
\begin{equation*}
  P_t(U_0,A)=\mathbb{P}(U(t,U_0)\in A)~~\text{for all} ~~U_0\in H,
\end{equation*}
and the Markov semigroup $\{P_t\}_{t\geq 0}$ associated to \eqref{MHD2.14} is defined as
\begin{equation*}
  P_t\Phi(U_0)=\int_H \Phi(U)P_t(U_0,\mathrm{d}U), ~~ P^*_t \mu(A)=\int_H P_t(U_0,A)\mu (\mathrm{d}U_0),
\end{equation*}
for every $\Phi: H\rightarrow\mathbb{R}$ and probability measurable $\mu$ on $H$.

Let $U=U(\cdot,U_0)$ be the solution of \eqref{MHD2.14} and let $d:=2\cdot|\mathcal{Z}_0|$. Then for any $\Phi\in C_b^1(H)$, and $\xi\in H$ we have
\begin{equation}\label{MHDP.1}
\nabla P_t\Phi(U_0)\cdot\xi
=\mathbb{E}(\nabla\Phi(U(t,U_0))\cdot J_{0,t}\xi),\quad t\geq0.
\end{equation}

Next, we explain how the estimate on $\nabla P_t\Phi$ can be translated to a control problem through the Malliavin integration by parts formula, which is closely related to the proof of the Proposition \ref{MHDpropo1.4}.

The crucial step in estimating $\|\nabla P_t\Phi(U_0)\|$ is to ``approximately remove" the gradient from $\Phi$ in \eqref{MHDP.1}. 
As such we seek to (approximately) identify $J_{0,t}\xi$ with a Malliavin derivative of some suitable random process and integrate by parts, in the Malliavin sense. 

For given $\ell\in\mathbb{S}$, $t>0$, 
let $\Psi(t,W)$ be a $\mathcal{F}_{\ell_t}^W$-measurable random variable. For $v\in L^2([0,\ell_t];\mathbb{R}^d)$, the Malliavin derivative of $\Psi$ in the direction $v$ is defined by
\begin{equation}
\mathcal{D}^{v}\Psi(t,W)
=\lim\limits_{\varepsilon\to0}\frac{1}{\varepsilon}
\left(\Psi(t,U_0,W+\varepsilon\int_0^\cdot v\mathrm{d}s)-\Psi(t,U_0,W)\right),
\end{equation}
where the limit holds almost surely (e.g. see \cite{MH-2006,MH-2011} for Hilbert space case).
In the definition of Malliavin derivative, the element $\ell$ is taken as
fixed. Then, $\mathcal{D}^{v}U_t$ satisfies the following equation:
\begin{equation}
\mathrm{d}\mathcal{D}^v U_t
=-A^{\alpha,\beta}\mathcal{D}^v U_t\mathrm{d}t
-\nabla B(U_t)\mathcal{D}^v U_t \mathrm{d}t
+Q_{b}\mathrm{d}\left(\int_0^{\ell_t}v_s \mathrm{d}s\right).
\end{equation}


Define that
$$\gamma_{u}=\inf\{t\geq0, S_t(\ell)\geq u\}.
$$ 
By the formula of constant variations or Fubini's theorem, for any $v\in L^2([0,\ell_t];\mathbb{R}^d)$, we deduce that 
\begin{equation}\label{MHDP.4}
\begin{aligned}
\mathcal{D}^v U_t
&=\int_0^t J_{r,t}Q_{b} \mathrm{d}\big(\int_0^{\ell_r}v_s\mathrm{d}s\big)
=\sum_{r\le t}J_{r,t}Q_{b} \int_{\ell_{r-}}^{\ell_r}v_s \mathrm{d}s\\
&=\sum_{r\le t}\int_{\ell_{r-}}^{\ell_{r}}J_{\gamma_{u},t}Q_{b} v_{u}\mathrm{d}u
=\int_{0}^{\ell_{t}}J_{\gamma_{u},t}Q_{b}  v_{u}\mathrm{d}u,
\end{aligned}
\end{equation}
the second equality is obvious, and we have used the fact that  $\gamma_u=r,u\in(\ell_{r-},\ell_{r})$ in the third equality.

By the Riesz representation theorem, 
\begin{equation}\label{MHDP.5}
\mathcal{D}_{u}^{j}U_t
=J_{\gamma_u,t}Q_{b}e_j,~~\text{for any}~ j=1,\ldots,d,~u\in[0,\ell_t].
\end{equation}
where $\mathcal{D}_u^jF:=(\mathcal{D}F)^j(u)$ is the $j$th component of $\mathcal{D}F$ evaluated at time $u$, and $\{e_j\}_{j=1,\ldots,d}$ is the standard basis of $\mathbb{R}^d$. 

Inspired by \eqref{MHDP.4}, for any $s\leq t$ and $\ell\in\mathbb{S}$, we define the random operator $\mathcal{A}_{s,t}:L^2([\ell_{s},\ell_{t}];\mathbb{R}^d)\rightarrow H$ by 
\begin{equation}\label{MHDP.7}
\mathcal{A}_{s,t}v
=\int_{\ell_s}^{\ell_t}J_{\gamma_u,t}Q_{b}v_u \mathrm{d}u,
~v\in L^2([\ell_s,\ell_t];\mathbb{R}^d).
\end{equation}
For any $s<t$, let $\mathcal{A}_{s,t}^*:H\rightarrow L^2([\ell_{s},\ell_{t}];\mathbb{R}^d)$ be the adjoint of $\mathcal{A}_{s,t}$, defined by
\begin{equation*}\label{MHDP.8}
\mathcal{A}_{s,t}^* \phi
=(\alpha_k^j\langle \phi,J_{\gamma_u,t}\sigma_k^j \rangle)_{k\in\mathcal{Z}_0,j\in\{0,1\} },~~\text{for any}~\phi \in H,~u\in[\ell_{s},\ell_{t}].
\end{equation*}
Direct computation shows that $\mathcal{A}_{s,t}$ satisfies the following equation
\begin{align}\label{MHDP.6}
\partial_t\mathcal{A}_{s,t}v + A^{\alpha,\beta}\mathcal{A}_{s,t}v + B(U_t,\mathcal{A}_{s,t}v) + B(\mathcal{A}_{s,t}v,U_t) = Q_b v_t, ~~\mathcal{A}_{s,s}v = 0.
\end{align}


For any $\xi=(\xi_1,\xi_2)\in H,t\geq s\geq0$, the Jacobian $J_{s,t}\xi$ corresponds to the unique solution of the variational equation
\begin{equation}\label{MHDP.2}
\partial_tJ_{s,t}\xi+A^{\alpha,\beta}J_{s,t}\xi+ B(U_t,J_{s,t}\xi)+B(J_{s,t}\xi,U_t)=0,\quad J_{s,s}\xi=\xi.
\end{equation}

Let $J_{s,t}^{(2)}:H\rightarrow \mathcal{L}(H,\mathcal{L}(H))$ be the second derivative of $U$ with respect to an initial condition $U_0$. Observe that for fixed $U_0$ in the directions of $\xi,\xi'\in H$. Then the function $\varrho=\varrho_t:= J_{s,t}^{(2)}(\xi,\xi')$ is the solution of
\begin{equation}\label{MHDP.3}
\begin{aligned}
&\partial_t \varrho_t
+A^{\alpha,\beta}\varrho_t+\nabla B(U_t)\varrho_t+\nabla B(J_{s,t}\xi)J_{s,t}\xi'=0,\\
& J_{s,s}^{(2)}(\xi,\xi')=0,
\end{aligned}
\end{equation}
where $\nabla B(\theta)\vartheta=B(\theta,\vartheta)+B(\vartheta,\theta)$.

For any $0\leq t\leq T$ and $\xi\in H$, let $K_{t,T}$ be the adjoint of $J_{t,T}$, i.e., $\rho_t^*:= K_{t,T}\xi=J^*_{t,T}\xi$ satisfies the following `backward' system
\begin{equation}\label{MHDbackward}
\partial_t\rho_t^{*}=A^{\alpha,\beta}\rho_t^*+(\nabla B(U(t)))^*\rho_t^*=-(\nabla F(U))^*\rho_t^*,
\quad\rho_T^*=\xi,
\end{equation}
where $\langle(\nabla B(U))^*\rho^*,\psi \rangle=\langle\rho^*,\nabla B(U)\psi\rangle$ and $\nabla B(U)\psi=B(U,\psi)+B(\psi,U)$.

We note that the noise is degenerate, rendering the Malliavin matrix non-invertible. Consequently, for any fixed value of  $\ell_t$, we cannot construct a function $v\in L^2([0,\ell_t];\mathbb{R}^d)$ that would generate the same infinitesimal shift in the solution as an arbitrary but fixed perturbation $\xi$ in the initial condition.

Defining $\rho_t:= J_{0,t}\xi-\mathcal{A}_{0,t}v$, we observe that it satisfies the following control problem:
\begin{equation}\label{MHD2.12}
\partial_t\rho_t+A^{\alpha,\beta}\rho_t+ B(U_t,\rho_t)+ B(\rho_t,U_t)=-Q_b v_t,\quad\rho_t=\xi.
\end{equation}

Our objective is to find a suitable $v$ such that $\rho_t\rightarrow 0$ as $t\rightarrow \infty$. To achieve this, the standard approach involves working with a regularized version $\mathcal M_{s,t}+\tilde{\beta}\mathbb{I} $ of the Malliavin matrix $\mathcal M_{s,t}$ (see \cite{MH-2006} for details), where $\tilde{\beta}>0$ is a constant, and $\mathbb{I} $ is the identity matrix.
However, building upon the ideas from \cite{Peng-2024}, we can identify an appropriate function $\phi$ satisfying $\mathbb{P}\Big(\inf_{\phi\in\mathcal{S}_{\varkappa,N}}
\langle\mathcal{M}_{0,\eta}\phi,\phi\rangle=0\Big)=0
$, as demonstrated in Proposition \ref{MHDpropo3.4}.

The object $\mathcal{M}_{s,t}:=\mathcal{A}_{s,t}\mathcal{A}_{s,t}^*: H\rightarrow H$, referred to as the Malliavin covariance matrix, plays an
important role in the theory of stochastic analysis. By a simple calculation, we have (c.f. \cite[Lemma 2.2]{Zhang-2014})
\begin{equation}\label{MHDMalliavin}
\langle\mathcal{M}_{s,t}\phi,\phi\rangle
=\sum_{k\in\mathcal{Z}_0,l\in\{0,1\}}(\alpha_k^l)^{2}\int_{s}^{t}\langle K_{r,t}\phi,\sigma_k^l\rangle^{2}\mathrm{d}\ell_{r}.
\end{equation}

If the invertibility of $\mathcal{M}_{s,t}$ can be proven, then we can choose appropriate direction $v$ with suitable bounds to determine an exact control of $\rho_{t}$, indicating that the Markov semigroup is smooth in finite time (i.e. it is strongly Feller).
As we know, for finite dimensional situations, the invertibility of the Malliavin matrix are well understood. However in infinite dimensions, this invertibility is considerably more difficult to determine and might not hold in general. 
Therefore, following the insights in \cite{MH-2006}, we can use a Tikhonov regularization of the Malliavin matrix to construct a control $v$, and corresponding $\rho_t$, which will be provided in detail in Section 4.

\subsection{Subordinated L\'{e}vy Noise}\label{MHDSUBBM}
We first construct the mathematical framework for subordinated Brownian motion $W_{S_t}$. Let $(\mathbb{W},\mathbb{H},\mathbb{P}^{\mu_{\mathbb{W}}})$ be the classical Wiener space, where
\begin{itemize}
    \item $\mathbb{W}$ is the space of all continuous functions from $\mathbb{R}^+$ to $\mathbb{R}^d$ with vanishing values at starting point 0;
    \item $\mathbb{H}\subseteq\mathbb{W}$ is the Cameron-Martin space consisting of absolutely continuous functions with square integrable derivatives;
    \item $\mathbb{P}^{\mu_{\mathbb{W}}}$ is the Wiener measure such that the coordinate process $W_t(\text{w}):=\text{w}_t$ forms a $d$-dimensional standard Brownian motion.
\end{itemize}

Let $\mathbb{S}$ be the space of all c\`{a}dl\`{a}g increasing functions $\ell:\mathbb{R}^+ \to \mathbb{R}^d$ with $\ell_0= 0$, equipped with the Skorohod metric and a probability measure $\mathbb{P}^{\mu_{\mathbb{S}}}$ so that the coordinate process $S_t(\ell):=\ell_t$ is a pure-jump subordinator with Lévy measure $\nu_S$.

Define the product probability space $(\Omega,\mathcal{F},\mathbb{P}):=(\mathbb{W}\times \mathbb{S},\mathcal{B}(\mathbb{W})\times \mathcal{B}(\mathbb{S}),\mathbb{P}^{\mu_{\mathbb{W}}}\times \mathbb{P}^{\mu_{\mathbb{S}}})$, and for $\omega=(\text{w},\ell)\in \mathbb{W}\times\mathbb{S}$, $L_t(\omega):=\text{w}_{\ell_t}$. Then $(L_t=W_{S_t})_{t\geq 0}$ is a $d$-dimensional pure jump Lévy process whose Lévy measure $\nu_L$ is given by
\begin{equation}\label{MHD1.8}
\nu_{L}(E)=\int_{0}^{\infty}\int_{E}(2\pi u)^{-\frac{d}{2}}e^{-\frac{|z|^{2}}{2u}}\mathrm{d}z\nu_{S}(\mathrm{d}u),\quad E\in\mathcal{B}(\mathbb{R}^{d}).
\end{equation}

The Poisson random measure $N_L(\mathrm{d}t,\mathrm{d}z)$ associated with \( L_t \) satisfies
\begin{equation*}
N_L((0,t]\times U)=\sum_{s\leq t}I_U(L_t-L_{t-}),U\in\mathcal{B}(\mathbb{R}^d\setminus\{0\}).
\end{equation*}
with compensated version $\tilde{N}_L(\mathrm{d}t,\mathrm{d}z)=N_L((0,t]\times U)-\nu_L(\mathrm{d}z)\mathrm{d}t$, the Lévy process admits the representation
\begin{equation*}
L_t=\int_0^t\int_{\mathbb{R}^d \setminus \{0\}}
z\tilde{N}_L(\mathrm{d}s,\mathrm{d}z).
\end{equation*}

Let $F=F(\mathrm{w},\ell)$ be a random variable defined on the probability space $(\Omega,\mathcal{F}, \mathbb{P})$. We define two conditional expectation operators:
$\mathbb{E}^{\mu_{\mathbb{W}}}[F]$ denotes the expectation with respect to $\mathrm{w}$ while keeping $\ell$ fixed
\begin{equation*}
\mathbb{E}^{\mu_\mathbb{W}}[F]=\int_\mathbb{W}F(\mathrm{w},\ell)
\mathbb{P}^{\mu_\mathbb{W}}(\mathrm{dw}).
\end{equation*}
The notation $\mathbb{E}^{\mu_{\mathbb{S}}}[F]$ is defined analogously by fixing $\mathrm{w}$ and integrating over $\ell$. We use $\mathbb{E}[F]$ to denote the expectation of $F$ under the measure $\mathbb{P}=\mathbb{P}^{\mu_{\mathbb{W}}}\times\mathbb{P}^{\mu_{\mathbb{S}}}$.

The filtration used in this paper is defined as follows
$$
\mathcal{F}_t:=\sigma(W_{S_t},S_s:s\leq t).
$$
For any fixed $\ell\in\mathbb{S}$ and positive number $a = a(\ell)$ which is independent of the Brownian motion $(W_t)_{t\geq 0}$, the filtration $\mathcal{F}_a^W$ is defined by
$$
\mathcal{F}_a^W:=\sigma(W_s:s\leq a).
$$

For any stopping time $\tau:\Omega\rightarrow[0,\infty]$ with respect to $\{\mathcal{F}_t\}$, we define the associated $\sigma$-algebra:
$$
\mathcal{F}_{\tau}=\{A\in\mathcal{F}:\forall t\geq0,A\cap\{\tau\leq t\}\in\mathcal{F}_t\}.
$$

We denote
\begin{equation}
\mathbb{Z}_+^2:=\left\{j=(j_1,j_2)\in\mathbb{Z}_0^2:
j_1>0\mathrm{~or~}j_1=0,~j_2>0\right\}.
\end{equation}
Let $\mathcal{Z}_0\subset \mathbb{Z}_+^2$ be a finite set representing forced directions in Fourier space. The driving noise is then the $d$-dimensional subordinated Brownian motion: $W_{S_t}:=(W_{S_t}^{k,l})_{k\in\mathcal{Z}_0,l\in\{0,1\}}$

Recall that the set $\mathcal{Z}_0\subset\mathbb{Z}^2_+$ represents the forced directions in Fourier space.
Note that there are many choices of finite $\mathcal{Z}_0$ that can accommodate our main result.
Therefore, let the set $\mathcal{Z}_0\subset\mathbb{Z}^2_*$ be a generator, if any element of $\mathbb{Z}^2_*$ is a finite linear combination of elements of $\mathcal{Z}_0$ with integer coefficients. 
For the remainder of our analysis, we assume that the following two conditions have been satisfied.

\textbf{Condition 2.1} The set $\mathcal{Z}_0\subset \mathbb{Z}_*^2$ appeared in \eqref{MHDsigma} is a finite, symmetric (i.e., $-\mathcal{Z}_0=\mathcal{Z}_0)$ generator that contains at least two non-parallel vectors $m$ and $n$ such that $|m|\neq|n|$.

The set 
\begin{equation*}
  \mathcal{Z}_0=\{(0,1), (0,-1), (1,1), (-1,-1)\}\subset\mathbb{Z}_*^2:=\mathbb{Z}^2\backslash\{(0,0)\}
\end{equation*}
is an example satisfying this condition.

For $n\geq0$,
\begin{equation*}
\begin{aligned}
 \mathcal{Z}_n:=\{k+\ell\mid k\in\mathcal{Z}_{n-1},\ell\in\mathcal{Z}_0,\langle k,\ell^\perp\rangle\neq0,|k|\neq|\ell|\},
\end{aligned}
\end{equation*}
where $l^{\bot}=(-l_2,l_1)$ for any $l=(l_1,l_2)$ and $\langle\cdot,\cdot\rangle$ denotes the inner product on $\mathbb{R}^2$.

\textbf{Condition 2.2} Assume that $\nu_S$ satisfies
\begin{equation*}
\int_0^\infty(e^{\zeta u}-1)\nu_S(\mathrm du)<\infty~\text{for some}~\zeta>0
\end{equation*}
and 
\begin{equation*}
  \nu_S((0,\infty))=\infty.
\end{equation*}

\subsection{Statement of main results}
Benefit from the preparation and understanding of the system mentioned above, we are now listing the main results of our article.
The ergodicity of invariant measures can usually be obtained by proving the irreducibility and e-property of Markov process, or strong Feller property, or asymptotic strong Feller property, refer to \cite{prato-1996,MH-2006,Kapica-2012,Komorowski-2010}. In this section, we will establish e-property and irreducibility to obtain the ergodicity for the system \eqref{MHD2.14}, these proofs will be presented in Section 5.
Actually, we first need to finish some preparatory work about the truncation function and give some estimates for the solutions of \eqref{MHD2.14}, which is different from \cite{MH-2006}, and that is placed in the Section 3. 

\begin{proposition}[e-property]\label{MHDpropo1.4} Let the Condition $2.1$ and Condition $2.2$ hold, then the Markov semigroup $\left\{ P_{t}\right\}_{t\geq 0}$ has the $e$-property, that is, for any bounded and Lipschitz continuous function $\Phi$, $U_0\in H$ and $\varepsilon > 0$, there exists $\delta > 0$ such that
$|P_t\Phi( U_0^{\prime })-P_t\Phi(U_0)|<\varepsilon$, $\forall t\geq 0$ and $U_0^{\prime }$ with $\| U_0^{\prime }- U_0\| < \delta$.
\end{proposition}

\begin{proposition}[Weak  Irreducibility]\label{MHDirre}  For any $\mathcal{C}, \gamma > 0$,  there exists a time $T=T(\mathcal{C},\gamma)>0$ such that
\begin{equation*}
\inf\limits_{\|U_0\|\leq\mathcal{C}}P_T(U_0,\mathcal{B}_\gamma)>0,
\end{equation*}
where $\mathcal{B} _\gamma = \{ U\in H:\| U\| \leq \gamma \}$.
\end{proposition}

\begin{theorem}[Ergodicity]\label{MHDergodicity}
Let the Condition $2.1$ and Condition $2.2$ hold, then the stochastic fractional MHD equation \eqref{MHD2.14} with highly degenerate L\'{e}vy noise admits a unique invariant measure $\mu^*$, i.e., $\mu^*$ is a unique probability measure on $H$ such that $P_t^*\mu^*=\mu^*$ for every $t\geq0$.
\end{theorem}

\section{Moment estimates on $U_t$, $J_{s,t}\xi$, $J_{s,t}^{(2)}(\xi,\xi')$, $\mathcal{M}_{s,t}$}
In this section, we establish crucial estimates for  $J_{s,t}$, $J_{s,t}^{(2)}$, $\mathcal{A}_{s,t}$, $\mathcal{A}_{s,t}^*$, $\mathcal{M}_{s,t}$ and their Malliavin derivatives, which will play fundamental roles in our subsequent analysis. 

However, for the stochastic fractional MHD equations \eqref{MHD1.1}, the regularity requirements the parameters, $\alpha > 1$ and $\beta > 1$, are necessary to compensate for the complexity of the advective operator $B$. These constraints are achieved through careful interpolation techniques and weighted estimates. The lemmas presented in this section illustrate the key aspects of this approach.

\begin{lemma}\label{MHDlemma2.1}
For any $U_0\in H$, let $U_t = U(t, U_0)$ be the unique solution of \eqref{MHD2.14} with initial value $U_0$. Then
there exists $\nu^*$ such that
for any $s\geq0$ and $\nu=\min\{\nu_1,\nu_2\}\in  (0, \nu^*]$, there exists $C_1,C_2>0$ depending on the parameters $\nu,\{\alpha_j\}_{j\in\mathcal Z_0}, \nu_S,d$, such that
\begin{equation}\label{MHD2.6}
\begin{aligned}
\mathbb{E}\left[\|U_{t}\|^{2}\right] & \leq e^{-\nu t}\|U_{0}\|^{2}+C_1, 
\end{aligned}\end{equation}
and
\begin{equation}\label{MHD2.4}
\begin{aligned}
\nu \mathbb{E}\left[\int_{0}^{t}\|U_{s}\|_{1}^{2}ds\right]
\leq\|U_{0}\|^{2}+C_{2}t,\quad\forall \,t\geq0.
\end{aligned}\end{equation}
\end{lemma}
\begin{proof}
Let $\nu_L$ be the intensity measure defined in \eqref{MHD1.8}. By \cite[Lemma 2.1]{Peng-2024}, we get that
\begin{equation*}
\int_{|z|\leq1}|z|^2 \nu_L(\mathrm{d}z)
+\int_{|z|\geq1}|z|^n \nu_L(\mathrm{d}z)
<\infty,\quad \forall\, n\geq 2.
\end{equation*}
By applying It\^{o}'s formula to $\|u_t\|^2$ and $\|b_t\|^2$, we obtain
\begin{equation}\label{MHD2.3}
\begin{aligned}
  &\mathrm{d}\|u\|^2+2\nu_1\|\Lambda^\alpha u\|^2\mathrm{d}t =2\langle \tilde{B}(b,b),u\rangle \mathrm{d}t,
    \end{aligned}
\end{equation}
and
\begin{equation}\label{MHD2.5}
\begin{aligned}
     \mathrm{d}\|b\|^2+2\nu_2\|\Lambda^\beta b\|^2\mathrm{d}t=&2\langle \tilde{B}(b,u),b\rangle \mathrm{d}t+2\int_{z\in \mathbb{R}^{d}\setminus \{0\}}\langle b,Q_b z\rangle \tilde{N}_L(\mathrm{d}t,\mathrm{d}z)\\
     &
     +\int_{z\in\mathbb{R}^d\setminus\{0\}}
     \|Q_b z\|^2N_L(\mathrm{d}t,\mathrm{d}z).
  \end{aligned}
\end{equation}
Set $C=\int_{z\in\mathbb{R}^d\setminus\{0\}}\|Q_b z\|^2\nu_L(\mathrm{d}z)$, which is a constant depending on $\{\alpha_k^m\}_{k\in \mathcal{Z}_0,m\in\{0,1\}},\nu_S,d$, then combining \eqref{MHD2.3} and \eqref{MHD2.5}, we can derive
\begin{equation*}
\mathbb{E}\left[\|U_t\|^2\right]
+2\nu\mathbb{E}\left[\int_0^t\|U_s\|_{1}^2
\mathrm{d}s\right]
\leq\|U_0\|^2+Ct,
\end{equation*}
the above equation implies the derivation of \eqref{MHD2.6} and \eqref{MHD2.4}. This proof is completed.
\end{proof}

Next we introduce some stopping times to build new preliminary estimates.

Let $\eta_0=0$ and $\mathfrak{B}_0=\sum\limits_{k\in\mathcal{Z}_0,m\in\{0,1\}}(\alpha_k^m)^2$. For any $\kappa>0$, $n\in \mathbb{N}$, $\ell\in \mathbb{S}$ and $\nu$, we define 
\begin{equation}\label{MHD2.7}
\eta=\eta(\ell)=\eta_1(\ell)=\inf\{t\ge0:\nu t-8\mathfrak{B}_0\kappa\ell_t>1
\},
\end{equation}
and
\begin{equation}\label{MHD2.8}
\eta_n=\eta_n(\ell)=\inf\{t\geq\eta_{n-1}, \nu(t-\eta_{n-1})-8\mathfrak{B}_0\kappa(\ell_t-\ell_{\eta_{n-1}})>1\}.
\end{equation}

For the solution of equation \eqref{MHD2.14} and the stopping times $\eta_n $(with respect to $\mathcal{F}_t$), we obtain the following moment estimates. Given the numerous constants in the remaining part of this section, we will use the following convention: the letter $C$ is a positive constant depending on $\nu$, $\{\alpha_k^m\}_{k\in \mathcal{Z}_0,m\in\{0,1\}},\nu_S$ and $d=2\cdot|\mathcal{Z}_0|$, the letter $C_{\kappa}$ is a positive constant depending on $\kappa$ and $\nu$, $\{\alpha_k^m\}_{k\in\mathcal{Z}_0,m\in\{0,1\}},\nu_S$, $d$, and $C_{n,\kappa}$ is a positive constant depending on $n,\kappa$ and $\nu$, $\{\alpha_k^m\}_{k\in\mathcal{Z}_0,m\in\{0,1\}},\nu_S$, $d$.

\begin{lemma}\label{MHDlemma2.4}
For $\xi\in H$, assume $J_{s,t}\xi=(J_{s,t}^1 \xi,J_{s,t}^2 \xi)\in H_1^0\times H_2^0$. For $\nu=\min\{\nu_1,\nu_2\}>0$ and $0<s<t$, there exists a constant $C$ only depending on $\nu$, we have the following estimates 
\begin{align}
&\sup\limits_{t\in[s,T]}\|J_{s,t}\xi\|^{2}
\leq C\|\xi\|^2 e^{C\int_s^t \|U_r\|_1^2\mathrm{d}r}, \label{MHD2.16}\\
&\int_s^t \left(\|\Lambda^\alpha J^1_{s,r}\xi\|_1^2
+\|\Lambda^\beta J^2_{s,r}\xi\|_1^2 \right)\mathrm{d}r
\leq C\|\xi\|^2 e^{C\int_s^t \|U_r\|_1^2\mathrm{d}r}, 
\label{MHD2.17}\\
&\sup\limits_{t\in[s,T]}
\|J^{(2)}_{s,t}(\xi,\xi')\|^2
\leq C\|\xi\|^2\|\xi'\|^2 e^{C\int_s^t \|U_r\|_1^2\mathrm{d}r}. \label{MHD2.18}
\end{align}
Moreover, for each $\tau\leq T$ and any $\nu>0$, there exists $C=C(\nu, T-\tau)$ such that
\begin{equation}\label{MHD2.19}
\begin{aligned}
\|J_{s,T}\xi\|^{2}
\leq& C\exp\Big\{\frac{\nu\kappa}{40}\int_{s}^{T}
\|U_{r}\|_{1}e^{-\nu(T-r)+8\mathfrak{B}_{0}\kappa
(\ell_{T}-\ell_{r})}\mathrm{d}r\\
&+C_{\kappa}\int_{s}^{T}
e^{4\nu(T-r)-32\mathfrak{B}_{0}\kappa(\ell_{T}-\ell_{r})}
\mathrm{d}r\Big\}\|\xi\|^{2},
\end{aligned}
\end{equation}
where $C$ is taken from \eqref{MHD2.16}-\eqref{MHD2.18}, and $C_\kappa$ is a constant depending on $\kappa, \nu$.
\end{lemma}
\begin{proof}
By using the interpolation inequality, we can get 
\begin{equation}
\begin{aligned}
\mathrm{d}\|J_{s,t}\xi\|^2 &= -2\langle A^{\alpha\beta}J_{s,t}\xi, J_{s,t}\xi\rangle \mathrm dt - 2\langle B(U_t, J_{s,t}\xi), J_{s,t}\xi\rangle \mathrm dt - 2\langle B(J_{s,t}\xi, U_t), J_{s,t}\xi\rangle \mathrm dt \\
&= -2\langle A^{\alpha\beta}J_{s,t}\xi, J_{s,t}\xi\rangle \mathrm dt - 2\langle B(J_{s,t}\xi, U_t), J_{s,t}\xi\rangle \mathrm d t \\
&\leq -2\|\Lambda^{\alpha}J_{s,t}^1\xi\|^2 \mathrm dt- 2\|\Lambda^{\beta}J_{s,t}^2\xi\|^2 \mathrm dt+ C\|U_t\|_{H^1} \cdot \left[\|\Lambda^{\alpha}J_{s,t}^1\xi\| + \|\Lambda^{\beta}J_{s,t}^2\xi\| \right] \cdot \|J_{s,t}\xi\| \mathrm dt\\
&\leq -\|\Lambda^{\alpha}J_{s,t}^1\xi\|^2 \mathrm dt- \|\Lambda^{\beta}J_{s,t}^2\xi\|^2 \mathrm dt+ C\|U_t\|_{H^1}^2 \cdot \|J_{s,t}\xi\|^2 \mathrm dt ,
\end{aligned}
\end{equation}
which can deduce \eqref{MHD2.16} and \eqref{MHD2.17}.

Given $U_0\in H$, applying Young's and interpolation inequalities to \eqref{MHDP.3}, then for any $\xi,\xi'\in H$, the second derivative $\varrho_t:=(\varrho_t^1,\varrho_t^2)\in H_1^0 \times H_2^0$ satisfies
\begin{align*}
&\,\, \partial_t \|\varrho_t\|^2 + 2\|\Lambda^\alpha \varrho_t^1\|^2+2\|\Lambda^\beta \varrho_t^2\|^2 \\
&\leq C\left[\|\Lambda^\alpha \varrho_t^1\| + \|\Lambda^\beta \varrho_t^2\|\right]\|\varrho_t\| \cdot \|U_t\|_{1} \\
&\quad +C\left[\|\Lambda^\alpha \varrho_t^1\| + \|\Lambda^\beta \varrho_t^2\|\right]\cdot \left[\|\Lambda^\alpha J_{s,t}^1\xi\| + \|\Lambda^\beta J_{s,t}^2\xi\|\right] \cdot \|J_{s,t}\xi'\| \\
&\quad +C\left[\|\Lambda^\alpha \varrho_t^1\| + \|\Lambda^\beta \varrho_t^2\|\right]
\left[\|\Lambda^\alpha J_{s,t}^1\xi'\| + \|\Lambda^\beta J_{s,t}^2\xi'\| \right]  \cdot \|J_{s,t}\xi\| \\
&\leq \|\Lambda^\alpha \varrho_t^1\|^2 + \|\Lambda^\beta \varrho_t^2\|^2
+C\|\varrho_t\|^2 \|U_t\|_{1}^2 
+ \left[\|\Lambda^\alpha J_{s,t}^1\xi\|^2 + \|\Lambda^\beta J_{s,t}^2\xi\|^2\right] \cdot \|J_{s,t}\xi'\|^2 \\
&\quad + \left[\|\Lambda^\alpha J_{s,t}^1\xi'\|^2 + \|\Lambda^\beta J_{s,t}^2\xi'\|^2 \right]  \cdot \|J_{s,t}\xi\|^2 .
\end{align*}
Applying the chain rule to $\|J_{s,t}^{(2)}(\xi,\xi')\|^2$, with the help of \eqref{MHD2.16} and \eqref{MHD2.17}, one arrives at
\begin{equation*}
\begin{aligned}
\|J_{s,t}^{(2)}(\xi,\xi')\|^{2}&\leq Ce^{C\int_{s}^{t}\|U_{r}\|_{1}^{2}\mathrm{d}r}
\left\{\int_{s}^{t}\left[\|\Lambda^\alpha J_{s,r}^1\xi\|^2 + \|\Lambda^\beta J_{s,r}^2\xi\|^2\right]  
\mathrm{d}r \sup\limits_{r\in[s,t]}\|J_{s,r}\xi'\|^2\right.\\
&\quad\quad\quad\quad\left. +\int_{s}^{t}\left[\|\Lambda^\alpha J_{s,r}^1\xi'\|^2 + \|\Lambda^\beta J_{s,r}^2\xi'\|^2\right]  
\mathrm{d}r \sup\limits_{r\in[s,t]}\|J_{s,r}\xi\|^2
\right\}\\
&\leq Ce^{C\int_{s}^{t}\|U_{r}\|_{1}^{2}\mathrm{d}r}
\|\xi'\|^{2}\|\xi\|^{2},
\end{aligned}
\end{equation*}
which shows \eqref{MHD2.18}. 
Moreover, an application of Young's inequality combined with equation \eqref{MHD2.16} yields the desired result \eqref{MHD2.19}. The proof is completed.
\end{proof}

Recall that $P_N$ is the orthogonal projection from $H$ into $H_N=span\{\sigma_k^l,\psi_k^l:0<|k|\leq N,l\in\{0,1\}\}$ and $Q_N=I-P_N$. For any $N\in\mathbb{N}$, $t\geq 0$ and $\xi\in H$, denote $\xi_t^Q=Q_N J_{0,t}\xi$, $\xi_t^P=P_N J_{0,t}\xi$ and $\xi_t=J_{0,t}\xi$.

\begin{lemma}\label{MHDlemma2.6}
For any $t\geq 0$ and $\xi\in H$, assume $\xi_t^Q=(\xi_t^{Q(1)},\xi_t^{Q(2)})\in H_1^0\times H_2^0$, one has
\begin{equation}
\|\xi_t^Q\|^2\leq\exp\{-\nu N^2 t\}\|\xi\|^2+\frac{C\|\xi\|^2}{\sqrt{N}}\exp\{C\|U_0\|^2+C_2 t \}
,
\end{equation}
where $C$ is a constant depending on $\nu=\min\{\nu_1,\nu_2\}$.
\end{lemma}
\begin{proof}
By chain rule, one has
\begin{align*}
\frac{1}{2}\frac{\mathrm{d}}{\mathrm{d}t}\|\xi_t^Q\|^2
&=\langle A^{\alpha,\beta}\xi_t,\xi_t^Q \rangle
-\langle B(\xi_t,U_t),\xi_t^Q \rangle
-\langle B(U_t,\xi_t),\xi_t^Q \rangle\\
&\leq -\nu_1\|\Lambda^\alpha \xi_t^{Q(1)}\|^2
-\nu_2\|\Lambda^\beta \xi_t^{Q(2)}\|^2
+C\|\xi_t\|\|U_t\|_1(\|\Lambda^\alpha \xi_t^{Q(1)}\|
+\|\Lambda^\beta \xi_t^{Q(2)}\|)\\
&\leq -\frac{1}{2}(\nu_1\|\Lambda^\alpha \xi_t^{Q(1)}\|^2
+\nu_2\|\Lambda^\beta \xi_t^{Q(2)}\|^2)
+C\|\xi_t\|^2\|U_t\|_{1}^2.
\end{align*}
Due to $\|\Lambda^\alpha \xi_t^{Q(1)}\|^2\geq
\| \xi_t^{Q(1)}\|_1^2\geq N^2\|\xi_t^{Q(1)}\|^2$, $\alpha>1$ (it also holds for $\Lambda^\beta \xi_t^{Q(2)}$), we can deduce that
\begin{align}\label{MHD2.5}
\|\xi_t^Q\|^2
&\leq \exp\{-\nu N^2 t\}\|\xi\|^2
+C\int_0^t \exp\{-\nu N^2(t-s)\}
\left(\|\xi_s\|^2\|U_s\|_{1}^2\right)\mathrm{d}s \notag \\
&\leq \exp\{-\nu N^2 t\}\|\xi\|^2+C\|\xi\|^2e^{C\int_0^t \|U_r\|_1^2\mathrm{d}r }
\left(\exp\{-4\nu N^2(t-s)\}\right)^{\frac{1}{4}}\left(\int_0^t \|U_s\|_1^{\frac{4}{3}}\mathrm{d}s\right)^{\frac{3}{4}} \notag\\
&\leq \exp\{-\nu N^2 t\}\|\xi\|^2+C\|\xi\|^2\exp\{C\|U_0\|^2+C_2 t \}
\left(\exp\{-4\nu N^2(t-s)\}\right)^{\frac{1}{4}}, \notag
\end{align}
where we used the result of \eqref{MHD2.6} and \eqref{MHD2.16}.
\end{proof}

\begin{lemma}\label{MHDlemma 2.7}
Assume that $\xi_0^P=0$ and $\xi_t^P=(\xi_t^{P(1)},\xi_t^{P(2)})\in H_1^0\times H_2^0$, then for any $t\geq 0$,
\begin{equation*}
\begin{aligned}
\|\xi_t^P\|^2&\leq C\|\xi\|^2\exp\{C\|U_0\|^2+C_2 t \}\left(\exp\{-\nu N^2 t\}+\frac{C}{\sqrt{N}}\right)
\left(\|U_0\|^2+C_2 t\right),
\end{aligned}
\end{equation*}
where $C$ is a constant depending on $\nu=\min\{\nu_1,\nu_2\}$. Furthermore, combining the above inequality with Lemma \ref{MHDlemma2.6}, for any $\xi\in H$ and $t\geq 0$, we have
\begin{equation*}
\begin{aligned}
&\|J_{0,t}Q_{N}\xi\|^{2} \leq C\|\xi\|^2\left(\exp\{-\nu N^2 t\}+\frac{C}{\sqrt{N}}\right)\exp\{C\|U_0\|^2+C_2 t \}
\left(\|U_0\|^2+C_2 t\right).
\end{aligned}
\end{equation*}
\end{lemma}
\begin{proof}
According to the interpolation inequality, there holds
\begin{equation*}
\begin{aligned}
&\langle B(\xi_t^P,U_t),\xi_t^P \rangle
+\langle B(\xi_t^Q,U_t),\xi_t^P \rangle
+\langle B(U_t,\xi_t^Q),\xi_t^P \rangle\\
\leq& C\|\xi_t^P\|_1\|U_t\|(\nu_1\|\Lambda^\alpha \xi_t^{P(1)}\|+\nu_2\|\Lambda^\beta \xi_t^{P(2)}\|)
+C\|\xi_t^Q\|\|U_t\|_{1} (\nu_1\|\Lambda^\alpha \xi_t^{P(1)}\|
+\nu_2\|\Lambda^\beta \xi_t^{P(2)}\|)\\
\leq & \frac{1}{2}(\nu_1\|\Lambda^\alpha \xi_t^{P(1)}\|^2
+\nu_2\|\Lambda^\beta \xi_t^{P(2)}\|^2)+C\|\xi_t^Q\|^2\|U_t\|_1^2
+C\|\xi_t^P\|_1^2\|U_t\|^2.
\end{aligned}
\end{equation*}
Applying the chain rule to $\|\xi_t^P\|^2$, we find that 
\begin{align*}
\|\xi_{t}^{P}\|^{2}&
\leq
C\exp\left\{C\int_0^t \|U_s\|^2\mathrm{d}s\right\}
\int_{0}^{t}\|U_{s}\|_{1}^{2}\|\xi_{s}^Q\|^{2}
\mathrm{d}s \\
&\leq
C\exp\left\{C\int_0^t \|U_s\|^2\mathrm{d}s\right\}
\sup\limits_{s\in[0,t]}\|\xi_{s}^Q\|^{2}
\int_{0}^{t}\|U_{s}\|_{1}^{2}\mathrm{d}s \\
&\leq C\exp\left\{C\int_0^t \|U_s\|^2\mathrm{d}s\right\}
\times \left(\exp\{-\nu N^2 t\}\|\xi\|^2+\frac{C\|\xi\|^2}{\sqrt{N}}\exp\{C\|U_0\|^2+C_2 t \}\right)\\
&~~~\times
\left(\|U_0\|^2+C_2 t\right)\\
&\leq C\|\xi\|^2\exp\{C\|U_0\|^2+C_2 t \}\left(\exp\{-\nu N^2 t\}+\frac{C}{\sqrt{N}}\right)
\left(\|U_0\|^2+C_2 t\right).
\end{align*}
The above inequality implies the desired result.
\end{proof}




In this paper, we denote by the notation $\mathbb{E}_{U_0}$ the expectation under the measure $\mathbb{P}$ with respect to solutions to \eqref{MHD2.14} with initial condition $U_0$.

\begin{lemma}\label{MHDmoment}
There exists a constant $\kappa_0\in(0,\nu]$ only depending on $\nu$, $\{\alpha_k^m\}_{k\in\mathcal{Z}_0,m\in\{0,1\}},\nu_S$ and $d=2\cdot|\mathcal{Z}_0|$ such that the following statements hold:

$(1)$ For any $\kappa\in(0,\kappa_0]$ and the stopping time $\eta$ defined in \eqref{MHD2.7}, we have
\begin{equation}\label{MHD2.9}
\mathbb{E}^{\mu_{\mathbb{S}}}[\exp\{10\nu\eta\}]\leq C_\kappa,
\end{equation}
hence,
\begin{equation}\label{MHD2.10}
\mathbb{E}[\exp\{10\nu\eta\}]\leq C_\kappa.
\end{equation}

$(2)$ For any $\kappa \in (0,\kappa_{0}]$, almost all $\ell\in \mathbb{S}$ (under the measure $\mathbb{P}^{\mu_{\mathbb{S}}}$) and the stopping times $\eta_{k}$ defined in \eqref{MHD2.7} and \eqref{MHD2.8}, we have
\begin{equation}\label{MHD3.11}
\begin{aligned}
\mathbb{E}^{\mu_{\mathbb{W}}}\Big[\exp&\Big\{ 
\kappa\|U_{\eta_k}\|^2-\kappa\|U_{\eta_{k-1}}\|^2e^{-1} \\
&+ \nu\kappa\int_{\eta_{k-1}}^{\eta_{k}}
e^{-\nu(\eta_{k}-s)+8\mathfrak{B}_{0}\kappa(\ell_{\eta_{k}}-\ell_{s})}
\|U_{s}\|_{1}^{2}\mathrm{d}s \\
&-\kappa\mathfrak{B}_{0}(\ell_{\eta_{k}}-\ell_{\eta_{k-1}})\Big\}
\Big|\mathcal{F}_{\ell_{\eta_{k-1}}}^{W}\Big]\leq C.
\end{aligned}
\end{equation}
Moreover,
the following statements hold:
\begin{equation}
\begin{aligned}
\mathbb{E}\bigg[\exp&\bigg\{\kappa\|U_{\eta_{k}}\|^{2}-\kappa\|U_{\eta_{k-1}}\|^{2}
e^{-1}\\
&+\nu\kappa\int_{\eta_{k-1}}^{\eta_{k}}
e^{-\nu(\eta_{k}-s)+8\mathcal{B}_{0}\kappa(\ell_{\eta_{k}}-\ell_{s})}
\|U_{s}\|_{1}^{2}\mathrm{d}s\\
&-\kappa\mathfrak{B}_{0}(\ell_{\eta_{k}}-\ell_{\eta_{k-1}})\Big\}
\Big|\mathcal{F}_{\eta_{k-1}}\Big]\leq C,
\end{aligned}
\end{equation}
where C is the constant appearing in \eqref{MHD3.11}.

$(3)$ For any $\kappa\in (0,\kappa_0]$ and $k\in \mathbb{N}$, one has
\begin{equation}
\mathbb{E}\Big[\exp\big\{\kappa\|U_{\eta_{k+1}}\|^{2}\big\}
\Big|\mathcal{F}_{\eta_{k}}\Big]\leq C_{\kappa}\exp\big\{\kappa e^{-1}\|U_{\eta_{k}}\|^{2}\big\}.
\end{equation}

$(4)$ For any $\kappa\in(0,\kappa_{0}]$, there exists a constant $C_{\kappa}>0$ 
such that for any $n\in\mathbb{N}$ and $U_{0}\in H$, one has
\begin{equation}\label{MHDmoment.4}
\mathbb{E}_{U_0}\left[\exp\{\kappa\sum_{i=1}^n\|U_{\eta_i}\|^2-C_\kappa n\}\right]
\leq e^{a\kappa\|U_0\|^2},
\end{equation}
where $a=\frac{1}{1-e^{-1}}$.

$(5)$ For any $\kappa\in(0,\kappa_0]$, $U_0\in H$ and $n\in\mathbb{N}$, one has 
\begin{equation}\label{MHDmoment.5}
\mathbb{E}_{U_0}\left[\sup_{s\in[0,\eta]}\|U_s\|^{2n}\right]\leq C_{n,\kappa}(1+\|U_0\|^{2n}).
\end{equation}
\end{lemma}

To derive the estimates of Lemma \ref{MHDmoment}, we first provide some preparatory work and knowledge.

For $\kappa>0$, $\varepsilon\in(0,1]$ and $\ell\in \mathbb{S}$, set
\begin{equation*}
\ell_t^\varepsilon=\frac1\varepsilon\int_t^{t+\varepsilon}\ell_s\mathrm{d}s
+\varepsilon t,
\end{equation*}
and
\begin{equation*}
\eta^{\varepsilon}=\eta^{\varepsilon}(\ell):=\inf \{t\geq0:\nu t-8\mathfrak{B}_{0}\kappa\ell_{t}^{\varepsilon}>1\}.
\end{equation*}
Recall that $\ell$ is a c\`{a}dl\`{a}g increasing function from $\mathbb{R}^+$ to $\mathbb{R}^+$ with $\ell_0=0$, it is easy to see that the following lemma is valid.

\begin{lemma}\label{MHDlemma A0}
For $\ell \in \mathbb{S}$,

$(i)$ $\ell_{\cdot }^{\varepsilon }: [ 0, \infty ) \to [ 0, \infty )$ is continuous and strictly increasing;

$(ii)$ for any $t\geq0$,  $\ell_{t}^{\varepsilon}$ strictly decreases to $\ell_t$ as $\varepsilon$ decreases to 0.
\end{lemma}

Referring to \cite[Lemma A.2]{Peng-2024}, the following moment estimates hold for the stopping times $\eta^{\varepsilon}$ and $\eta$.

\begin{lemma}\label{MHDlemma A1}
  There exists a constant $\tilde{\kappa}_0>0$ such that for any $\kappa\in(0,\tilde{\kappa}_0]$, 
\begin{equation}
\sup_{\varepsilon\in(0,1]}\mathbb{E}^{\mu_{\mathbb{S}}}
\left[\exp\{10\nu\eta^{\varepsilon}\}\right]\leq C_\kappa,
\end{equation}
and
\begin{equation}
\mathbb{E}^{\mu_{\mathbb{S}}}\left[\exp\{10\nu\eta\}\right]\leq C_\kappa.
\end{equation}
\end{lemma}

For any $\kappa\in(0,\tilde{\kappa}_0]$, set 
\begin{equation}\label{MHDS1}
  \mathbb{S}_1=\{\ell\in\mathbb{S}:\eta^1(\ell)<\infty\}.
\end{equation}
We have the following lemma.

\begin{lemma}\label{MHDlemma A2}
We claim that $\mathbb{P}^{\mu_{\mathbb{S}}}$. For any $\ell\in\mu_{\mathbb{S}} $, the following statements hold.

$(1)$ For any $\varepsilon\in(0,1)$, $\eta^{\varepsilon}<\eta^1<\infty$ and $\nu\eta^\varepsilon-8\mathfrak{B}_0\kappa
\ell_{\eta^\varepsilon}^\varepsilon=1$;

$(2)$ $\eta^{\varepsilon}$ strictly decreases to $\eta$ as $\varepsilon$ decreases to 0;

$(3)$ $\ell_{\eta^\varepsilon}^\varepsilon$ strictly decreases to $\ell_{\eta}$ as $\varepsilon$ decreases to 0;

$(4)$ $\nu\eta-8\mathfrak{B}_0\kappa\ell_{\eta}=1$;

$(5)$ $
\lim\sup\limits_{\varepsilon\to0}\int_{0}^{\eta^{\varepsilon}}
\exp\{-\nu(\eta^{\varepsilon}-s)
+8\mathfrak{B}_{0}\kappa(\ell_{\eta^{\varepsilon}}^{\varepsilon}
-\ell_{s}^{\varepsilon})\}\mathrm{d}\ell_{s}^{\varepsilon}\leq\ell_{\eta}.
$
\end{lemma}
\begin{proof}
  The proof of Lemma \ref{MHDlemma A2} is similar to \cite{Peng-2024}.
\end{proof}

Let $\mathcal{H}_0=\text{span}\{\sigma_k^m:k\in\mathcal{Z}_0,m\in\{0,1\}\}$ and $D([0,\infty);\mathcal{H}_0)$ be the space of all c\`{a}dl\`{a}g functions taking values in $\mathcal{H}_0$. Keeping in mind that $d=2\cdot|\mathcal{Z}_0|<\infty$, it is well-known that, for any $U_0\in H$ and $f\in D([0,\infty);\mathcal{H}_0)$, there exists a unique solution $\Psi(U_{0},f)\in C([0,\infty);H)\cap L_{loc}^{2}([0,\infty);V)$ to the following partial differential equation (PDE):
\begin{equation*}
\begin{aligned}
\Psi(U_{0},f)(t)
=&U_{0}-\int_{0}^{t}A^{\alpha,\beta}(\Psi(U_{0},f)(s)+f_{s})\mathrm{d}s
-\int_{0}^{t}B(\Psi(U_{0},f)(s)
+f_{s},\Psi(U_{0},f)(s)
+f_{s})\mathrm{d}s.
\end{aligned}
\end{equation*}
In the above statements, $V=\{h\in H:\|h\|_1<\infty\}$.

Define $\Psi(U_0,f)
:=(\tilde{u},\tilde{b})^T
=(u,b)^T-(0,f)^T$
\begin{equation}\label{MHDA.2}
\begin{aligned}
&\mathrm{d}\tilde{u}+\nu_1(-\Delta)^{\alpha} \tilde{u}\mathrm{d}t+\tilde{B}(\tilde{u},\tilde{u})\mathrm{d}t
-\tilde{B}(\tilde{u}+f,\tilde{u}+f)\mathrm{d}t
=0,\\
&\mathrm{d}\tilde{b}+\nu_2(-\Delta)^{\beta} (\tilde{b}+g)\mathrm{d}t +\tilde{B}(\tilde{u},\tilde{b}+f)\mathrm{d}t
-\tilde{B}(\tilde{b}+f,\tilde{u})\mathrm{d}t=0.
\end{aligned}
\end{equation}

We denote $q_t^{\varepsilon}=Q_{b}(W_{\ell_{t}^{\varepsilon}} - W_{\ell_{0}^{\varepsilon}}),q_{t}= Q_{b}W_{\ell_{t}},V_{t}^{\varepsilon} = \Psi(U_{0},q^{\varepsilon})(t)$ and $V_{t} = \Psi(U_{0},q)(t)$. It is easy to see that $V_t+q_t$ is the unique solution $U_t$ to \eqref{MHD2.14}, i.e., $U_t=V_t+q_t$, and for any $\ell\in \mathbb{S}$ and $\varepsilon\in(0,1]$,  $U_t^{\varepsilon}=V_t^{\varepsilon}+q_t^{\varepsilon}$ is the solution of the following PDE:
\begin{equation*}
 U_t^{\varepsilon}=U_0-\int_0^t[A^{\alpha,\beta}U_s^{\varepsilon}
+B(U_s^{\varepsilon},U_s^{\varepsilon})]\mathrm{d}s
+Q_{b}(W_{\ell_t^{\varepsilon}}-W_{\ell_0^{\varepsilon}})
\end{equation*}

Recall $\mathbb{S}_1$ introduced in \eqref{MHDS1}. We have

\begin{lemma}
For any $\ell\in \mathbb{S}_1$, the following statements hold:
\begin{equation}\label{MHDA.3}
\begin{gathered}
\qquad \operatorname*{lim}_{\varepsilon\to 0}
\int_{0}^{\eta^{\varepsilon}}e^{-\nu(\eta^{\varepsilon}-s)
+8\mathfrak{B}_{0}\kappa(\ell_{\eta^{\varepsilon}}^{\varepsilon}
-\ell_{s}^{\varepsilon})}\|U_{s}^{\varepsilon}\|_{1}^{2}\mathrm{d}s \\
=\int_{0}^{\eta}e^{-\nu(\eta-s)
+8\mathfrak{B}_{0}\kappa(\ell_{\eta}-\ell_{s})}
\|U_{s}\|_{1}^{2}\mathrm{d}s, 
\end{gathered}
\end{equation} 
and
\begin{equation}\label{MHDA.4}
\begin{aligned}
\lim_{\varepsilon\to0}\|U_{\eta^\varepsilon}^\varepsilon
-U_\eta\|^2=0.
\end{aligned}
\end{equation}
\end{lemma}
\begin{proof}
To prove this lemma, we first need some a priori estimates for $\Psi$.
  
By the chain rule, there exists a constant $C>0$ dependent on $\nu_1,\nu_2$, such that, for any $U_0\in H, f\in D([0,\infty);\mathcal{H}_0)$ and $t\geq 0$,
\begin{equation*}
\begin{aligned}
&\,\, \|\Psi(U_{0},f)(t)\|^2
+2\int_{0}^{t}\langle A^{\alpha,\beta}(\Psi(U_{0},f)(s)+f_s),\Psi(U_{0},f)(s)\rangle\mathrm{d}s\\
&\leq \|U_{0}\|^2+2\int_{0}^{t}\langle B(\Psi(U_{0},f)(s)
+f_{s},\Psi(U_{0},f)(s)+f_{s}),\Psi(U_{0},f)(s)\rangle\mathrm{d}s\\
&\leq\|U_{0}\|^2+2\int_{0}^{t}\langle B(\Psi(U_{0},f)(s)
,f_{s}),\Psi(U_{0},f)(s)\rangle+\langle B(f_{s},
f_{s}),\Psi(U_{0},f)(s)\rangle\mathrm{d}s\\
&\leq \|U_{0}\|^2+C\int_{0}^{t}(\|\Psi(U_{0},f)(s)\|
\|f_{s}\|_1+\|f_{s}\|
\|f_{s}\|_1)(\|\Lambda^{\alpha}\tilde{u}(s)\|
+\|\Lambda^{\beta}\tilde{b}(s)\|)\mathrm{d}s\\
&\leq \|U_{0}\|^2+C\int_{0}^{t}\left[\|\Psi(U_{0},f)(s)\|^2
\|f_{s}\|_1^2+ \|f_{s}\|^2
\|f_{s}\|_1^2\right]\mathrm{d}s
+\frac{1}{2}\int_{0}^{t}\left[\nu_1\|\Lambda^{\alpha}\tilde{u}(s)\|^2
+\nu_2\|\Lambda^{\beta}\tilde{b}(s)\|^2\right]\mathrm{d}s.
\end{aligned}
\end{equation*}
Note that
\begin{equation*}
\begin{aligned}
&\left|-2\langle A^{\alpha,\beta}f_s,\Psi(U_{0},f)(s)\rangle\right|
\leq \nu_2\|\Lambda^{\beta} f_s^{(2)}\|^2+\nu_2\|\Lambda^{\beta} \tilde{b}(s)\|^2.
\end{aligned}
\end{equation*}
Then we arrive at 
\begin{equation}\label{MHDA.10}
\begin{aligned}
&\sup\limits_{t\in[0,T]}\|\Psi(U_{0},f)(t)\|^2
+\int_{0}^{T}\left[\nu_1\|\Lambda^{\alpha} \tilde{u}(s)\|^2+\nu_2\|\Lambda^{\beta} \tilde{b}(s)\|^2\right]
\mathrm{d}s\\
&\leq C\left[\|U_0\|^2+\int_0^T \left( \|f_s\|^2\|f_s\|_1^2
+\|\Lambda^{\beta} f^{(2)}_s\|^2 \right) \mathrm{d}s\right]e^{C\int_0^T\|f_s\|_1^2\mathrm{d}s}.
\end{aligned} 
\end{equation}

For any $f^1,f^2 \in D([0,\infty);\mathcal{H}_0)$, let $\Psi^1(t)=\Psi(U_0,f^1)(t)=(\tilde{u}^1(t),\tilde{b}^1(t))^T$ and $\Psi^2(t)=\Psi(U_0,f^2)(t)=(\tilde{u}^2(t),\tilde{b}^2(t))^T$ be two solutions of system \eqref{MHDA.2}, simplifying the notation. Using similar arguments as above, it follows that 
\begin{align*}
& \|\Psi^{1}(t)-\Psi^{2}(t)\|^{2}
+2\int_{0}^{t}\left(\nu_1\|\Lambda^{\alpha}(\tilde{u}^{1}(s)-\tilde{u}^2(s))
\|^{2}
+\nu_2\|\Lambda^{\beta}(\tilde{b}^1(s)-\tilde{b}^{2}(s))
\|^{2}\right)\mathrm{d}s \\
\leq&2\int_{0}^{t}\langle B(\Psi^{1}(s)+f_{s}^{1},\Psi^{1}(s)+f_{s}^{1})
-B(\Psi^{2}(s)+f_{s}^{2},\Psi^{2}(s)+f_{s}^{2}),
\Psi^{1}(s)-\Psi^{2}(s)\rangle \mathrm{d}s\\
=& 2\int_{0}^{t}\langle B(\Psi^{1}(s)+f_{s}^{1},\Psi^{1}(s)+f_{s}^{1})
-B(\Psi^{1}(s)+f_{s}^{1},\Psi^{2}(s)+f_{s}^{2}),
\Psi^{1}(s)-\Psi^{2}(s)\rangle \mathrm{d}s \\
&+2\int_{0}^{t}\langle B(\Psi^{1}(s)+f_{s}^{1},\Psi^{2}(s)+f_{s}^{2})
-B(\Psi^{2}(s)+f_{s}^{2},\Psi^{2}(s)+f_{s}^{2}),
\Psi^{1}(s)-\Psi^{2}(s)\rangle \mathrm{d}s \\
\leq & C\int_{0}^{t}\|\Psi^{1}(s)-\Psi^{2}(s)\|
\|\Psi^{1}(s)+f_{s}^{1}\|\|f^{1}(s)-f^{2}(s)\|_2\mathrm{d}s \\
&+C\int_{0}^{t}\|\Psi^{1}(s)-\Psi^{2}(s)\|_{1}
\|\Psi^{2}(s)+f_{s}^{2}\|_{1}\|\Psi^{1}(s)-\Psi^{2}(s)\|\mathrm{d}s \\
&+C\int_0^t\|f^1(s)-f^2(s)\|_1
\|\Psi^2(s)+f_s^2\|_1\|\Psi^1(s)-\Psi^2(s)\|\mathrm{d}s \\
\leq& \int_{0}^{t}\left(\nu_1\|\Lambda^{\alpha}(\tilde{u}^{1}(s)-\tilde{u}^2(s))
\|^{2}
+\nu_2\|\Lambda^{\beta}(\tilde{b}^1(s)-\tilde{b}^{2}(s))
\|^{2}\right)\mathrm{d}s\\
&+C\left(1+\sup_{s\in[0,t]}\|\Psi^{1}(s)+f_{s}^{1}\|^{2}\right)
\int_{0}^{t}\|f^{1}(s)-f^{2}(s)\|^{2}\mathrm{d}s 
\\
&  +C\int_{0}^{t}\left(1+\|\Psi^{2}(s)+f_{s}^{2}\|_{1}^{2}\right)
\|\Psi^{1}(s)-\Psi^{2}(s)\|^{2}\mathrm{d}s.
\end{align*}

Rearranging terms and using the Gronwall lemma, we obatin that, for any $T \geq 0$,
\begin{equation}\label{MHDA.11}
\begin{aligned}
&\sup_{t\in[0,T]}\|\Psi^{1}(t)-\Psi^{2}(t)\|^{2}
+\int_{0}^{T}\left(\nu_1\|\Lambda^{\alpha}(\tilde{u}^1(s)-\tilde{u}^2(s))
\|^{2}
+\nu_2\|\Lambda^{\beta}(\tilde{b}^1(s)-\tilde{b}^2(s))
\|^{2}\right)\mathrm{d}s\\
\leq& \exp\left\{C\int_{0}^{T}\left(1+\|\Psi^{2}(s)+f_{s}^{2}\|_{1}^{2}\right)\mathrm{d}s\right\}\\
&\times
C\left(1+\sup_{s\in[0,T]}\|\Psi^{1}(s)+f_{s}^{1}\|^{2}\right)
\int_{0}^{T}\|f^{1}(s)-f^{2}(s)\|^{2}\mathrm{d}s.
\end{aligned}
\end{equation}

For any $(\mathrm{w},\ell)\in \mathbb{W}\times\mathbb{S}$, from the definitions of $q_t^{\varepsilon}$ and $q_t$, it is easy to see that for any $T\geq0$,
\begin{equation}\label{MHDA.15}
\sup\limits_{\varepsilon\in(0,1]}
\sup\limits_{t\in[0,T]}
\left(\|q_t^\varepsilon(\text{w},\ell)\|
+\|q_t(\text{w},\ell)\|\right)
\leq C\sup\limits_{t\in[0,\ell_{T+1}+T]}\|\text{w}_t\|
<\infty,
\end{equation}
and
\begin{equation}\label{MHDA.16}
\lim_{\varepsilon\to0}
\int_0^T\|q_t^\varepsilon(\mathrm{w},\ell)
-q_t(\mathrm{w},\ell)\|^2\mathrm{d}t=0.
\end{equation}
Combining \eqref{MHDA.15}-\eqref{MHDA.16} with \eqref{MHDA.10}-\eqref{MHDA.11}, there exists a constant $C$ dependent on $U_0, T$ and  $\sup\limits_{t\in[0,\ell_{T+1}+T]}\|\text{w}_t\|$, such that the following estimates hold
\begin{equation}\label{MHDA.12}
\begin{aligned}
&\sup_{\varepsilon\in(0,1]}\Big(\sup_{t\in[0,T]}
\|U_{t}^{\varepsilon}\|^{2}
+\int_{0}^{T}\|U_{t}^{\varepsilon}\|_{1}^{2}\mathrm{d}t\Big)(\mathrm{w},\ell)+\Big(\sup_{t\in[0,T]}\|U_{t}\|^{2}
+\int_{0}^{T}\|U_{t}\|_{1}^{2}\mathrm{d}t\Big)(\mathrm{w},\ell)
\leq  C,
\end{aligned}
\end{equation}
and 
\begin{equation}\label{MHDA.13}
\lim\limits_{\varepsilon\to0}\Big(\sup\limits_{t\in[0,T]}
\|V_t^\varepsilon-V_t\|^2
+\int_0^T\|U_t^\varepsilon-U_t\|_1^2\mathrm{d}t\Big)(\mathrm{w},\ell)
=0.
\end{equation}
Notice that for any $\ell\in \mathbb{S}_1$ and $\text{w}\in \mathbb{W}$,
\begin{equation}\label{MHDA.14}
\begin{aligned}
\|U_{\eta^{\varepsilon}}^{\varepsilon}-U_{\eta}\|
\leq&\|V_{\eta^{\varepsilon}}^{\varepsilon}-V_{\eta}\|
+\|q_{\eta^{\varepsilon}}^{\varepsilon}-q_{\eta}\| \\
\leq&\|V_{\eta^{\varepsilon}}^{\varepsilon}-V_{\eta^{\varepsilon}}\|
+\|V_{\eta^{\varepsilon}}-V_{\eta}\|
+\|q_{\eta^{\varepsilon}}^{\varepsilon}-q_{\eta}\| \\
\leq&\sup_{t\in[0,\eta^{1}]}\|V_{t}^{\varepsilon}-V_{t}\|
+\|V_{\eta^{\varepsilon}}-V_{\eta}\|
+\|Q_b(W_{\ell_{\eta^{\varepsilon}}^{\varepsilon}}
-W_{\ell_{0}^{\varepsilon}})-Q_b W_{\ell_{\eta}}\|.
\end{aligned}
\end{equation}
Applying Lemma \ref{MHDlemma A0}, Lemma \ref{MHDlemma A1}, \eqref{MHDA.12}-\eqref{MHDA.14}, and the fact that $V_t$ is continuous in $H$, we can deduce that \eqref{MHDA.3} and \eqref{MHDA.4}. This proof is completed.
\end{proof}

\begin{lemma}\label{MHDlemma A4}
There exists a positive constant $C$,
such that, for any $\kappa\in(0,\tilde{\kappa}_0],\varepsilon\in(0,1]$, $\ell\in\mathbb{S}_1$, there holds
\begin{equation}
\begin{aligned}
&\mathbb{E}^{\mu_{\mathbb{W}}}\left[\exp\left\{
\kappa\|U_{\eta^{\varepsilon}}^{\varepsilon}\|^2
-\kappa\|U_0\|^2 e^{-\nu\eta^{\varepsilon}
+8\mathfrak{B}_{0}\kappa \ell_{\eta^{\varepsilon}}^{\varepsilon}}\right.\right.\\
&+\kappa^*\int_{0}^{\eta^{\varepsilon}}
e^{-\nu(\eta^{\varepsilon}-s)
+8\mathfrak{B}_{0}\kappa(\ell_{\eta^{\varepsilon}}^{\varepsilon}
-\ell_{s}^{\varepsilon})}\left(\nu_1\|\Lambda^{\alpha}X_{s}^{\varepsilon}\|^{2}
+\nu_2\|\Lambda^{\beta}Y_{s}^{\varepsilon}\|^{2}\right){\mathrm{d}}s\\
&\left.\left.-\kappa\mathfrak{B}_{0}\int_{0}^{\eta^{\varepsilon}}
e^{-\nu(\eta^{\varepsilon}-s)
+8\mathfrak{B}_{0}\kappa(\ell_{\eta^{\varepsilon}}^{\varepsilon}
-\ell_{s}^{\varepsilon})}\mathrm{d}\ell_{s}^{\varepsilon}\right\}\right]
\leq C.
\end{aligned}
\end{equation}
\end{lemma}
\begin{proof}
  Now we fix $\kappa\in (0,\tilde{\kappa}], \varepsilon\in (0,1]$ and $\ell\in \mathbb{S}_1$.

Let $\gamma^{\varepsilon}$ be the inverse function of $\ell^{\varepsilon}$. By a change of variable, for $t\geq\ell_0^{\varepsilon}$, we set $V_t^\varepsilon=(X_t^\varepsilon,Y_t^\varepsilon)
:=(u_{\gamma_t^{\varepsilon}}^{\varepsilon}
,b_{\gamma_t^{\varepsilon}}^{\varepsilon})$, then it satisfies the following stochastic equation
\begin{equation*}
\begin{aligned}
 V_t^\varepsilon&=V_0
-\int_{\ell_0^\varepsilon}^t\left[ A^{\alpha,\beta} V_s^\varepsilon
+ B(V_s^\varepsilon,V_s^\varepsilon)
\right]
\dot{\gamma}_s^\varepsilon \mathrm{d}s
+ Q_{b}(W_t-W_{\ell_0^\varepsilon}).
\end{aligned}
\end{equation*}
Applying It\^{o} formula, we have
\begin{equation*}
\begin{aligned}
\mathrm{d}\|V_t^\varepsilon\|^2=
-2\langle A^{\alpha,\beta} V_t^\varepsilon,V_t^\varepsilon\rangle \dot{\gamma}_t^\varepsilon \mathrm{d}t
+2\langle V_t^\varepsilon, Q_b \mathrm{d}W_t\rangle +\mathfrak B_0 \mathrm{d}t,
\end{aligned}
\end{equation*}
and
\begin{equation*}
\begin{aligned}
&\mathrm{d}\left[\kappa\|V_{t}^{\varepsilon}\|^{2}
e^{\nu\gamma_{t}^{\varepsilon}-8\mathfrak{B}_{0}\kappa t} \right] \\
\leq& e^{\nu\gamma_{t}^{\varepsilon}-8\mathfrak{B}_{0}\kappa t}
\big[-2\kappa\nu_1\|\Lambda^{\alpha}X_{t}^{\varepsilon}\|^{2}
\dot{\gamma}_{t}^{\varepsilon}\mathrm{d}t-2\kappa \nu_2
\|\Lambda^{\beta}Y_{t}^{\varepsilon}\|^{2}
\dot{\gamma}_{t}^{\varepsilon}\mathrm{d}t 
+2\kappa\langle V_{t}^{\varepsilon},Q_{b} \mathrm{d}W_{t}\rangle+\kappa\mathfrak{B}_{0}\mathrm{d}t\big] \\
&+\kappa
\|V_{t}^{\varepsilon}\|^{2}
e^{\nu\gamma_{t}^{\varepsilon}-8\mathfrak{B}_{0}\kappa t}\big(\nu\dot{\gamma}_{t}^{\varepsilon}-8\mathfrak{B}_{0}\kappa\big)\mathrm{d}t \\
\leq&  -\kappa e^{\nu\gamma_{t}^{\varepsilon}-8\mathfrak{B}_{0}\kappa t}
\big(\nu_1\|\Lambda^{\alpha}X_{t}^{\varepsilon}\|^{2}
+\nu_2\|\Lambda^{\beta}Y_{t}^{\varepsilon}\|^{2}\big)
\dot{\gamma}_{t}^{\varepsilon}\mathrm{d}t
+2\kappa e^{\nu\gamma_{t}^{\varepsilon}-8\mathfrak{B}_{0}\kappa t}\langle V_{t}^{\varepsilon},Q_{b} \mathrm{d}W_{t}\rangle\\
&+\kappa \mathfrak{B}_{0}e^{\nu\gamma_{t}^{\varepsilon}-8\mathfrak{B}_{0}\kappa t}\mathrm{d}t -8\mathfrak{B}_{0}\kappa^2 e^{\nu\gamma_{t}^{\varepsilon}-8\mathfrak{B}_{0}\kappa t}\|V_{t}^{\varepsilon}\|^{2}\mathrm{d} t,
\end{aligned}
\end{equation*}
where we used the Poincar\'{e} inequality $\|V_{t}^{\varepsilon}\|^2\leq \| X_{t}^{\varepsilon}\|_1^2+\|Y_{t}^{\varepsilon}\|_1^2
\leq\|\Lambda^\alpha X_{t}^{\varepsilon}\|^2+\|\Lambda^\beta Y_{t}^{\varepsilon}\|^2 $.
Therefore, the above inequality can be rewritten as  
\begin{equation}\label{MHDA.7}
\begin{aligned}
&\kappa\|V_{t}^{\varepsilon}\|^{2}
+\kappa \int_{\ell_{0}^{\varepsilon}}^{t} e^{-\nu(\gamma_{t}^{\varepsilon}-\gamma_{s}^{\varepsilon})
+8\mathfrak{B}_{0}\kappa(t-s)}
(\nu_1\|\Lambda^{\alpha}X_{s}^{\varepsilon}\|^{2}
+\nu_2\|\Lambda^{\beta}Y_{s}^{\varepsilon}\|^{2}) 
\dot{\gamma}_{s}^{\varepsilon}\mathrm{d}s\\
\leq& \kappa\|U_0\|^2e^{-\nu\gamma_{t}^{\varepsilon}
+8\mathfrak{B}_{0}\kappa t}
+\kappa\mathfrak{B}_{0}\int_{\ell_{0}^{\varepsilon}}^{t}
e^{-\nu(\gamma_{t}^{\varepsilon}-\gamma_{s}^{\varepsilon})
+8\mathfrak{B}_{0}\kappa(t-s)}\mathrm{d}s+\tilde{M}_{t},
\end{aligned}
\end{equation}
where
\begin{equation*}
\begin{aligned}
\tilde{M}_{t}
=\tilde{M}_{t}^{\kappa,\varepsilon}
=&2\kappa\int_{\ell_{0}^{\varepsilon}}^{t}e^{-\nu(\gamma_{t}^{\varepsilon}-\gamma_{s}^{\varepsilon})
+8\mathfrak{B}_{0}\kappa(t-s)}\langle V_{s}^{\varepsilon},Q_{b} \mathrm{d}W_{s}\rangle
\\
&-8\mathfrak{B}_{0}\kappa^{2}
\int_{\ell_{0}^{\varepsilon}}^{t}\|V_{s}^{\varepsilon}\|^{2}
e^{-\nu(\gamma_{t}^{\varepsilon}-\gamma_{s}^{\varepsilon})
+8\mathfrak{B}_{0}\kappa(t-s)}\mathrm{d}s.
\end{aligned}
\end{equation*}

From \cite[Lemma A.5]{Peng-2024}, we know that
\begin{equation}\label{MHDA.5}
\mathbb{E}^{\mu_{\mathbb{W}}}\left[
\exp\{\tilde{M}_{\ell_{\eta^{\varepsilon}}^{\varepsilon}}\}\right]
\leq C.
\end{equation}
Replace $t$ in \eqref{MHDA.7} by $\ell_{\eta^{\varepsilon}}^{\varepsilon}$, it follows that
\begin{align}\label{MHDA.6}
&\mathbb{E}^{\mu_{\mathbb{W}}}
\bigg[\exp\bigg\{\kappa
\|V_{\ell_{\eta^{\varepsilon}}^\varepsilon}
^{\varepsilon}\|^{2}+\kappa\int_{\ell_{0}^{\varepsilon}}
^{\ell_{\eta^{\varepsilon}}^{\varepsilon}}e^{-(\eta^{\varepsilon}
-\gamma_{s}^{\varepsilon})
+8\mathfrak{B}_{0}\kappa(\ell_{\eta^{\varepsilon}}^{\varepsilon}-s)}
(\nu_1\|\Lambda^{\alpha}X_{s}^{\varepsilon}\|^{2}
+\nu_2\|\Lambda^{\beta}Y_{s}^{\varepsilon}\|^{2})\dot{\gamma}_{s}^{\varepsilon}\mathrm{d}s \notag\\
&~~~-\kappa\|U_0\|^2e^{-\eta^{\varepsilon}
+8\mathfrak{B}_{0}\kappa\ell_{\eta^{\varepsilon}}^{\varepsilon}}
-\kappa\mathfrak{B}_{0}\int_{\ell_{0}^{\varepsilon}}^{\ell_{\eta^{\varepsilon}}^{\varepsilon}}
e^{-(\eta^{\varepsilon}-\gamma_{s}^{\varepsilon})
+8\mathfrak{B}_{0}\kappa(\ell_{\eta^{\varepsilon}}^{\varepsilon}-s)}\mathrm{d}s\Big\} \bigg] \\
&\leq\mathbb{E}^{\mu_{\mathbb{W}}}
\bigg[\exp\left\{\tilde{M}_{\ell_{\eta^{\varepsilon}}^{\varepsilon}}\right\}\bigg].\notag
\end{align}
By using the fact that $X_{\ell_{\eta^{\varepsilon}}^{\varepsilon}}^{\varepsilon}
=u_{\eta^{\varepsilon}}^{\varepsilon},
Y_{\ell_{\eta^{\varepsilon}}^{\varepsilon}}^{\varepsilon}
=b_{\eta^{\varepsilon}}^{\varepsilon}$, and
$\gamma_{s}^{\varepsilon}|_{s=\ell_{r}^{\varepsilon}}=r$, we obtain
\begin{equation*}
\begin{aligned}
&\int_{\ell_{0}^{\varepsilon}}^{\ell_{\eta^{\varepsilon}}^{\varepsilon}}
e^{-(\eta^{\varepsilon}-\gamma_{s}^{\varepsilon})
+8\mathfrak{B}_{0}\kappa(\ell_{\eta^{\varepsilon}}^{\varepsilon}-s)}
(\nu_1\|\Lambda^{\alpha}X_{s}^{\varepsilon}\|^{2}
+\nu_2\|\Lambda^{\beta}Y_{s}^{\varepsilon}\|^{2})\dot{\gamma}_{s}^{\varepsilon}\mathrm{d}s\\
&~~~=\int_{0}^{\eta^{\varepsilon}}e^{-(\eta^{\varepsilon}-r)
+8\mathfrak{B}_{0}\kappa(\ell_{\eta^{\varepsilon}}^{\varepsilon}-\ell_{r}^{\varepsilon})}
(\nu_1\|\Lambda^{\alpha}u_{r}^{\varepsilon}\|^{2}
+\nu_2\|\Lambda^{\beta}b_{r}^{\varepsilon}\|^{2})\mathrm{d}r,
\\
&\int_{\ell_{0}^{\varepsilon}}^{\ell_{\eta^{\varepsilon}}^{\varepsilon}}
e^{-(\eta^{\varepsilon}-\gamma_{s}^{\varepsilon})
+8\mathfrak{B}_{0}\kappa(\ell_{\eta^{\varepsilon}}^{\varepsilon}-s)}\mathrm{d}s
=\int_{0}^{\eta^{\varepsilon}}e^{-(\eta^{\varepsilon}-r)
+8\mathfrak{B}_{0}\kappa(\ell_{\eta^{\varepsilon}}^{\varepsilon}
-\ell_{r}^{\varepsilon})}\mathrm{d}\ell_{r}^{\varepsilon}.
\end{aligned}\end{equation*}
Combining the above formulas with \eqref{MHDA.5}-\eqref{MHDA.6}, we can derive the desired result. This proof is completed.
\end{proof}

With the groundwork laid by the previous lemmas, we can derive \eqref{MHD2.9}-\eqref{MHDmoment.4} following the arguments in the proof of \cite[Appendix A]{Peng-2024}.
Finally, we turn to prove \eqref{MHDmoment.5}.
\begin{proof}[Proof of \eqref{MHDmoment.5}] Applying It\^{o}'s formula to $\|X_{s}^{\varepsilon}\|^{2n}+\|Y_{s}^{\varepsilon}\|^{2n}$ yields
\begin{align}\label{MHDA.8}
&\,\, \|X_{s}^{\varepsilon}\|^{2n}+\|Y_{s}^{\varepsilon}\|^{2n}+2 n \nu_1 \int^s_{\ell_{0}^{\varepsilon}}
\|X_{u}^{\varepsilon}\|^{2n-2}\|\Lambda^\alpha X_{u}^{\varepsilon}\|^{2}
\dot{\gamma}_{u}^{\varepsilon}\mathrm{d}u
+2 n \nu_2 \int^s_{\ell_{0}^{\varepsilon}}
\|Y_{u}^{\varepsilon}\|^{2n-2}\|\Lambda^\beta Y_{u}^{\varepsilon}\|^{2}
\dot{\gamma}_{u}^{\varepsilon}\mathrm{d}u \notag\\
&\leq \|u_{0}\|^{2n}+\|b_{0}\|^{2n}
+2 n\int_{\ell_{0}^{\varepsilon}}^{s}
\|Y_{u}^{\varepsilon}\|^{2n-2}\langle Y_{u}^{\varepsilon},
Q_b \mathrm{d}W_{u}\rangle+n\int_{\ell_{0}^{\varepsilon}}^{s}\|Y_{u}^{\varepsilon}\|^{2n-2}\mathfrak{B}_{0}\mathrm{d}u \\
&~~~
+2n(2n-1)\int_{\ell_{0}^{\varepsilon}}^{s}\|Y_{u}^{\varepsilon}\|^{2n-4}
\sum_{j\in\mathcal{Z}_0}\langle Y_{u}^{\varepsilon},\sigma_j^m\rangle^2 (\alpha_j^m)^2 \mathrm{d}u. \notag
\end{align}
We take expectations with respect to $\mathbb{P}^{\mu_{\mathbb{W}}}$ to \eqref{MHDA.8}, it has
\begin{equation*}
\begin{aligned}
\mathbb{E}^{\mu_{\mathbb{W}}}\left[\|X_{s}^{\varepsilon}\|^{2n}
+\|Y_{s}^{\varepsilon}\|^{2n}\right]
&\leq \|u_{0}\|^{2n}+\|b_{0}\|^{2n}
+C_n\int_{\ell_{0}^{\varepsilon}}^{s}\mathbb{E}^{\mu_{\mathbb{W}}}
\left[\|Y_{u}^{\varepsilon}\|^{2n-2}\right]\mathrm{d}u\\
&\leq \|u_{0}\|^{2n}+\|b_{0}\|^{2n}
+\frac{1}{2}\sup\limits_{u\in[\ell_0^\varepsilon,s]}
\mathbb{E}^{\mu_{\mathbb{W}}}
\left[\|Y_{u}^{\varepsilon}\|^{2n}\right]+C_n t^n.
\end{aligned}
\end{equation*}
Then we can get that, for any $t\geq \ell_0^{\varepsilon}$,
\begin{equation*}
\begin{aligned}
&\sup\limits
_{s\in[\ell_0^\varepsilon,t]}
\mathbb{E}^{\mu_{\mathbb{W}}}\left[\|X_{s}^{\varepsilon}\|^{2n}
+\|Y_{s}^{\varepsilon}\|^{2n}\right]
\leq2\| U_{0}\|^{2n}+C_n t^{n}.
\end{aligned}
\end{equation*}
Using \eqref{MHDA.8} and the Burkholder-Davis-Gundy inequality to obtain
\begin{align*}
&\,\,\mathbb{E}^{\mu_{\mathbb{W}}}\left[\sup\limits
_{s\in[\ell_0^\varepsilon,t]}\left(\|X_{s}^{\varepsilon}\|^{2n}
+\|Y_{s}^{\varepsilon}\|^{2n}\right)\right]\\
&\leq \|U_0\|^{2n}
+2 n\mathbb{E}^{\mu_{\mathbb{W}}}\left[\sup\limits
_{s\in[\ell_0^\varepsilon,t]}\left|\int_{\ell_{0}^{\varepsilon}}^{s}
\|Y_{u}^{\varepsilon}\|^{2n-2}\langle Y_{u}^{\varepsilon},
Q_b \mathrm{d}W_{u}\rangle\right|\right]
+C_n \mathbb{E}^{\mu_{\mathbb{W}}}\left[
\int_{\ell_{0}^{\varepsilon}}^{s}\|Y_{u}^{\varepsilon}\|^{2n-2}\mathrm{d}u\right]\\
&\leq \|U_0\|^{2n}
+C_n \left(\mathbb{E}^{\mu_{\mathbb{W}}}\left[\sup\limits
_{s\in[\ell_0^\varepsilon,t]}\left|\int_{\ell_{0}^{\varepsilon}}^{s}
\|Y_{u}^{\varepsilon}\|^{2n-2}\langle Y_{u}^{\varepsilon},
Q_b \mathrm{d}W_{u}\rangle\right|^2\right]\right)^{1/2}+C_n \mathbb{E}^{\mu_{\mathbb{W}}}\left[
\int_{\ell_{0}^{\varepsilon}}^{s}\|Y_{u}^{\varepsilon}\|^{2n-2}\mathrm{d}u\right]\\
&\leq \|U_0\|^{2n}
+C_n \left(\mathbb{E}^{\mu_{\mathbb{W}}}\left[\int_{\ell_{0}^{\varepsilon}}^{t}
\|Y_{u}^{\varepsilon}\|^{4n-2}\mathrm{d}u\right]\right)^{1/2}
+C_n \mathbb{E}^{\mu_{\mathbb{W}}}\left[
\int_{\ell_{0}^{\varepsilon}}^{s}\|Y_{u}^{\varepsilon}\|^{2n-2}\mathrm{d}u\right]\\
&\leq \|U_0\|^{2n}
+C_n \left[\| U_{0}\|^{4n-2}t+t^{2n}\right]^{1/2}
+C_n \left[ \| U_{0}\|^{2n-2}t+t^n\right]\\
&\leq C_n(1+t)\|U_0\|^{2n}+C_n(1+t^n).
\end{align*}
Following the argument in \cite[Appendix A.3]{Peng-2024}, we can obtain the desired result. This proof is completed.
\end{proof}

Following the analytical framework developed in \cite[Section 4.8]{MH-2006} or extended in \cite[Lemma A.6]{Foldes-2015}, we obtain the following result.
\begin{lemma}\label{MHDlemma2.9} For $0 \leq s < t,$
\begin{align*}
\|\mathcal{A}_{s,t}\|_{\mathcal{L}(L^2([\ell_s,\ell_t];\mathbb{R}^d);H)} 
\leq C \left( \int_{\ell_s}^{\ell_t} \|J_{\gamma_r,t}\|_{\mathcal{L}(H,H)}^2 \mathrm{d}r \right)^{1/2},
\end{align*}
holds for a constant C independent of $s,t$. Moreover, for any $\kappa>0$,
\begin{align*}
&\|\mathcal{A}_{s,t}^{*}  (\mathcal{M}_{s,t}+\kappa I)^{-1/2}\|_{\mathcal L(H,L^2([\ell_s,\ell_t];\mathbb{R}^d))} \leq 1,\\
&\|(\mathcal{M}_{s,t}+\kappa I)^{-1/2}\mathcal{A}_{s,t}\|_{\mathcal{L}(L^2([\ell_s,\ell_t];\mathbb{R}^d),H)} \leq 1,\\
&\|(\mathcal{M}_{s,t}+\kappa I)^{-1/2}\|_{\mathcal{L}(H,H)} 
\leq \kappa^{-1/2}.
\end{align*}
Here, $\mathcal L(X,Y)$ denotes the operator norm of the linear map between the given Hilbert spaces $X$ and $Y$.
\end{lemma}

By the Riesz representation theorem and \eqref{MHDP.4}, we have 
\begin{equation}\label{MHDriesz}
\mathcal{D}_{u}^{j}U_t
=J_{\gamma_u,t}Q_{b}e_j,~~\text{for any}~ j=1,\ldots,d,~u\in[0,\ell_t],
\end{equation}
where $\mathcal{D}_u^jF:=(\mathcal{D}F)^j(u)$ is the $j$th component of $\mathcal{D}F$ evaluated at time $u$, and $\{e_j\}_{j=1,\ldots,d}$ is the standard basis of $\mathbb{R}^d$. 

Recall that $\mathcal{D}$ is the Malliavin derivative. As in \cite{MH-2006}, let $A:\mathcal{H}_1\rightarrow\mathcal{H}_2$ be a random linear map between two Hilbert spaces, then we can use $\mathcal{D}_{s}^j A:\mathcal{H}_1\rightarrow\mathcal{H}_2$ to denote the random linear map given by
\begin{equation*}
(\mathcal{D}_{s}^j A)h:=\langle\mathcal{D}_{s}(Ah),e_j\rangle.
\end{equation*}
Following the approach in \cite[(4.29)]{MH-2006}, and in view of \eqref{MHDP.5}, we note that for any $0\leq s\leq t$, $j=\{1,\cdots,d\}$ and $u \in [0,\ell_t]$, the following conclusion holds:
\begin{equation}\label{MHD2.23}
\mathcal{D}_{u}^{j}J_{s,t}\xi
=\begin{cases}
J_{\gamma_u,t}^{(2)}(Q_{b}e_{j},J_{s,\gamma_u}\xi),~u\in [\ell_s,\ell_t],
\\[2ex]
J_{s,t}^{(2)}(J_{\gamma_u,s}Q_{b}e_{j},\xi),~\mathrm{if}~u<\ell_s.
\end{cases}
\end{equation}

\begin{lemma}\label{MHDlemma2.10}
For any $0\leq s<t$, the random operators $ J_{s,t}, \mathcal{A}_{s,t}, \mathcal{A}_{s,t}^{*}$ are differentiable in the Malliavin sense. Moreover, for any $r>0$, we have the following results:
\begin{align}
\|\mathcal{D}_{r}^{j} J_{0,\eta}\|_{\mathcal L(H,H)}
&\leq C_\kappa\exp\{C\int_0^\eta\|U_s\|_1^2 \mathrm{d}s\}, \label{MHD2.34}
\\
\|\mathcal{D}_{r}^{j}\mathcal{A}_{0,\eta}\|
_{\mathcal{L}(L^{2}([0,\ell_\eta];\mathbb{R}^{d}),H)}^{p}
&\leq C_\kappa(\eta+1)\exp\{C\int_0^\eta\|U_s\|_1^2 \mathrm{d}s\},\label{MHD2.35}
\\
\|\mathcal{D}_{r}^{j}\mathcal{A}_{0,\eta}^{*}\|
_{\mathcal{L}(H,L^{2}([0,\ell_\eta];\mathbb{R}^{d}))}^{p}
&\leq C_\kappa(\eta+1)\exp\{C\int_0^\eta\|U_s\|_1^2 \mathrm{d}s\},\label{MHD2.36}
\end{align}
where $C$ is the same constant as that appeared in Lemma \ref{MHDlemma2.4}, $C_\kappa$ is a constant depending on $\kappa, \nu , \{\alpha_j\}_{j\in \mathcal Z_0},d$.
\end{lemma}
\begin{proof}
The inequality \eqref{MHD2.34} follows directly from Lemma \ref{MHDlemma2.4} combined with result \eqref{MHD2.23}. For inequality \eqref{MHD2.35}, we employ Lemma \ref{MHDlemma2.4} combined with \eqref{MHD2.23}, applying the Cauchy-Schwarz inequality while observing the key property that
\begin{equation*}
\mathcal A_{0,\eta}v = \int_0^{\ell_\eta} J_{\gamma_u,\sigma}Q_b v(u)du, \quad \ell_\eta \leq \frac{\nu\eta}{8\mathfrak{B}_0\kappa}.
\end{equation*}
The inequality \eqref{MHD2.36} is a consequence of \eqref{MHD2.35}. This proof is completed.
\end{proof}

\section{The invertibility of the Malliavin matrix $\mathcal{M}_{0,t}$}

This section is devoted to establishing a key proposition concerning the spectral properties of eigenvectors with significant projections along unstable directions. Specifically, we demonstrate that such eigenvectors typically correspond to small eigenvalues. This result plays a fundamental role in ensuring the invertibility of the Malliavin matrix when restricted to the subspace spanned by unstable directions.
Given the finite-dimensional nature of this subspace in our framework, we are able to formulate a control problem leveraging the Malliavin integration by parts formula. This approach enables us to derive precise gradient estimates for the Markov semigroup, that are crucial for establishing the ergodic properties of the system.


\begin{proposition}\label{MHDpropo3.4} For any $\varkappa \in ( 0, 1] , N\in \mathbb{N}$ and $U_0\in H$, one has
\begin{equation}\label{MHD3.4}
\mathbb{P}\Big(\inf_{\phi\in\mathcal{S}_{\varkappa,N}}
\langle\mathcal{M}_{0,\eta}\phi,\phi\rangle=0\Big)=0,
\end{equation}
where $
\mathcal{S}_{\varkappa,N}=\{\phi\in H:\|P_{N}\phi \| \geq \varkappa , \| \phi \| = 1\}.
$
\end{proposition}

Note that Proposition \ref{MHDpropo3.4} is not sufficient for the proof of Proposition \ref{MHDpropo1.4} (e-property), we need to make a necessary supplement to the proposition. For $\varkappa\in(0,1]$, $U_0\in H$, $N\in\mathbb{N}$, $\mathfrak{R}>0$ and $\varepsilon>0$, let
\begin{equation}
X^{U_0,\varkappa,N}=\inf_{\phi\in\mathcal{S}_{\varkappa,N}}
\langle\mathcal{M}_{0,\eta}\phi,\phi\rangle.
\end{equation}
We denote
\begin{equation}
r(\varepsilon,\varkappa,\mathfrak{R},N)
=\sup_{\|U_0\|\leq\Re}\mathbb{P}(X^{U_0,\alpha,N}<\varepsilon).
\end{equation}
Let us recall that $\mathcal{M}_{0,t}$ denotes the Malliavin covariance matrix associated with the solution $U_t$ of equation \eqref{MHD2.14} with initial condition $U_0$. Consequently, $\mathcal{M}_{0,\eta}$
inherits dependence on the initial datum $U_0$. We will prove the above result after proving Proposition \ref{MHDpropo3.4}.

Prior to proceeding with the proof of Proposition \ref{MHDpropo3.4}, we require several preparatory results concerning Lévy processes and some technical estimates.

Throughout this section, we adopt the notation $\Delta f(s):=f(s)-f(s-)$ for the jump size of a function $f$ at times $s$. We work on the probability space
$(\Omega,\mathcal{F},\mathbb{P})$ introduced in Section 2, where $L_t=(W_{S_t}^{k,m})_{k\in \mathcal Z_0,m\in \{0,1\}}$ constitutes a $\mathbb{R}^d$-dimensional L\'{e}vy process characterized by its $\sigma$-finite intensity measure $\nu_L$.

The following lemma establishes a pure jump version of Theorem 7.1 in \cite{MH-2011}, extending the Wiener process results to our Lévy driven setting.
\begin{lemma}[\cite{Peng-2024}]\label{MHDlemma3.2}
If for some $\omega _0\in \Omega _0$, the following three conditions are satisfied:

$(1)$ $a( \omega _{0}) , b( \omega _{0}) \in [ 0, \infty )$ and $a( \omega _{0}) < b( \omega _{0})$;

$(2)$ $g_{i}( \omega _{0}, \cdot ) : [ a( \omega _{0}) , b( \omega _{0}) ] \rightarrow \mathbb{R} , 0\leq i\leq d$, are continuous functions;

$(3)$
\begin{equation}\label{MHD3.1}
g_0(\omega_0,r)+\sum_{i=1}^dg_i(\omega_0,r)\tilde{L}^i(\omega_0,r)=0,\quad\forall r\in[a(\omega_0),b(\omega_0)].
\end{equation}

Then
\begin{equation*}
g_i(\omega_0,r)=0,\quad\forall r\in[a(\omega_0),b(\omega_0)],\quad0\leq i\leq d.
\end{equation*}
\end{lemma}

Recall that the assumption in Condition 2.2: $\nu_S((0,\infty))=\infty$. For the process $S_t,t\geq0$, by \cite[Lemma 3.1]{Peng-2024}
we have the following lemma.
\begin{lemma}\label{MHDlemma3.3}
Let  $\tilde{\mathbb S}_0= \left\{\ell:  \{r:\Delta\ell_r:=\ell_r-\ell_{r-}>0\} \text{ is dense in }   [0,\infty)\right\} $ 
and  denote  $\tilde \Omega_0 =\mathbb W \times  \tilde{\mathbb S}_0.$
Then 
\begin{equation*}
\mathbb{P}^{\mu_{\mathbb{S}}}\left(  \tilde{\mathbb S}_0    \right)=1~~\text{and}~~\mathbb P(\tilde \Omega_0) =1.
\end{equation*}
\end{lemma}

We introduce an operation $[E_1,E_2]$, which is called as the Lie bracket of the two `vector field' $E_1$, $E_2$, specifically, for any Fr\'{e}chet differentiable $E_1, E_2:H\rightarrow H$, 
\begin{equation*}
[E_1,E_2](U):=\nabla E_2(U)E_1(U)-\nabla E_1(U)E_2(U).
\end{equation*}

Notice that the initial degenerate noise acts only on the magnetic equation, and the range of $Q_b$ is not sufficiently large, 
we are unable to `directly control' all low (unstable) modes that have not been dissipated by the diffusion, which hinders our work.
This requires careful estimation of several Lie brackets through iterative computations, for which we directly reference the results established in \cite{Peng-2020}. The proof proceeds via an inductive iteration scheme that involves the simultaneous application and distinction between two function spaces. This framework enables us to systematically handle progressively larger finite-dimensional subspaces in a continuous manner, moreover, the introduction of Lemma \ref{MHDlemma3.2} makes our analysis more concise.

Next, we conduct several Lie bracket computations and establish useful lemmas in preparation for the subsequent analysis. For a comprehensive derivation, interested readers may refer to \cite[Section 4]{Peng-2020}.

\textbf{Covering velocity direction.} For $u,\tilde{u}\in H_1^0=H_2^0$, denote $\mathbf{b}(u,\tilde{u}):=u\cdot\nabla\tilde{u}$. For any $k,l\in\mathbb{Z}^2$, $m,m'\in\{0,1\}$ and $U=(u,b)\in H_1^0\times H_2^0$, we introduce
\begin{equation*}\begin{aligned}
Y_k^{m}(U): & =[F(U),\sigma_{k}^{m}] \\
 & = A^{\alpha,\beta}\sigma_{k}^{m}+B(\sigma_{k}^{m},U)+B(U,\sigma_{k}^{m}), \end{aligned}\end{equation*}
 \begin{equation}\label{MHD4.1}
 \begin{aligned}
J_{k,l}^{m,m^{\prime}}(U): & =-[Y_{k}^{m}(U),\sigma_{l}^{m^{\prime}}] \\
 & = B(\sigma_{k}^{m},\sigma_{l}^{m^{\prime}})+B(\sigma_{l}^{m^{\prime}},\sigma_{k}^{m}) \\
 & =
\begin{pmatrix}
-\Pi\mathbf{b}(e_{k}^{m},e_{l}^{m^{\prime}})-\Pi\mathbf{b}(e_{l}^{m^{\prime}},e_{k}^{m}) \\
0
\end{pmatrix} .
\end{aligned}\end{equation}

In fact, $Y_k^{m}(U)$ and $J_{k,l}^{m,m^{\prime}}(U)$ are devised elaborately by calculation to guarantee that the following two lemmas hold.

We can generate suitable directions in the $u$ component.

\begin{lemma}\label{MHDlemma4.3}
Denote $a=\frac{\langle k,l^{\bot}\rangle}{|k||l|}$, and $k,l\in\mathbb{Z}_+^2$, then for some absolutely non-zero constant $c$ which is independent of $k,l$, the following inequalities hold:
\begin{align*}
J_{k,l}^{0,1}+J_{l,k}^{0,1} & = ac\frac{|l|^{2}-|k|^{2}}{|k+l|}\cdot\psi_{k+l}^{0}, \\
J_{k,l}^{0,1}-J_{l,k}^{0,1} & = ac\frac{-|l|^{2}+|k|^{2}}{|k-l|}\cdot\psi_{k-l}^{0}, \\
J_{k,l}^{1,1}+J_{l,k}^{0,0} & = ac\frac{|l|^{2}-|k|^{2}}{|k-l|}\cdot\psi_{k-l}^{1}, \\
J_{k,l}^{1,1}-J_{l,k}^{0,0} & = ac\frac{|l|^{2}-|k|^{2}}{|k+l|}\cdot\psi_{k+l}^{1}.
\end{align*}
\end{lemma}

\textbf{Covering magnetic direction.} 
Correspondingly, we establish the following notations for the b direction, which are likewise derived through iterative Lie bracket computations.

\begin{equation*}\begin{aligned}
\mathcal{Y}_{k}^{m}(U): & =\quad[F(U),\psi_{k}^{m}] \\
 & =\quad A^{\alpha,\beta}\psi_{k}^{m}+B(\psi_{k}^{m},U)+B(U,\psi_{k}^{m}) \end{aligned}\end{equation*} 
 \begin{equation}\label{MHD4.4}
 \begin{aligned}
Z_{k,l}^{m,m^{\prime}}:  &=\quad-\left[\mathcal{Y}_{k}^{m}(U),\sigma_{l}^{m^{\prime}}\right] \\
 & =\quad B(\psi_{k}^{m},\sigma_{l}^{m^{\prime}})
 +B(\sigma_{l}^{m^{\prime}},\psi_{k}^{m}) \\
 & =
\begin{pmatrix}
0 \\
 \\
\Pi[\mathbf{b}(e_{k}^{m},e_{l}^{m^{\prime}})
-\mathbf{b}(e_{l}^{m^{\prime}},e_{k}^{m})]
\end{pmatrix} .
\end{aligned}\end{equation}

The following lemma is the counterpart of Lemma \ref{MHDlemma4.3}.



\begin{lemma}\label{MHDlemma4.5}
Denote $a=\frac{\langle k,l^{\bot}\rangle}{|k||l|}$ and $k,l\in\mathbb{Z}_+^2$, then for some absolutely non-zero constant c which is independent of $k,l$ (It may changes from line to line), the following equalities hold:
\begin{align*}
 Z_{k,l}^{0,1}+ Z_{l,k}^{0,1}
&=ac|k-l|\sigma_{k-l}^0,\\
 Z_{k,l}^{0,1}- Z_{l,k}^{0,1}
&=ac|k+l|\sigma_{k+l}^0,\\
 Z_{k,l}^{1,1}+ Z_{k,l}^{0,0}
&=ac|k-l|\sigma_{k-l}^1,\\
 Z_{k,l}^{1,1}- Z_{k,l}^{0,0}
&=ac|k+l|\sigma_{k+l}^1.
\end{align*}
\end{lemma}

We denote that $\bar{U}=U-Q_b W_{S_t}$, then
\begin{equation}\label{MHDK.2}
 \partial_t \bar{U}=F(U)=F(\bar{U}+Q_b W_{S_t}),
~~~\bar{U}(0)=U_0.
\end{equation}

\begin{lemma}\label{MHDlemma0}
For any  $\omega=\mathrm{w}\times \ell\in $ $\tilde{\Omega}_0$,  the following conclusion 
\begin{equation}\label{MHDK.1}
\langle\mathcal{M}_{0,\eta}\phi,\phi\rangle=0~~
\Rightarrow~~\sup_{t\in[\eta/2,\eta]}|\langle
K_{t,\eta}\phi,\sigma_k^m\rangle|(\omega)=0,
\end{equation}
holds for all $\phi\in \mathcal{S}_{\varkappa,N}$, $k\in \mathcal{Z}_0$ and $m\in\{0,1\}$.
\end{lemma}

\begin{lemma}\label{MHDlemmaB} Recall that  $\Omega_{0}$ is   given by Lemma \ref{MHDlemma3.2}.
For $\omega\in\Omega_{0}$, the following  
\begin{equation}\label{MHDK.3}
\sup_{t\in[\eta/2,\eta]}|\langle
K_{t,\eta}\phi,\sigma_k^m\rangle|(\omega)=0~~
\Rightarrow~~\sup_{t\in[\eta/2,\eta]}|\langle
K_{t,\eta}\phi,Y_k^m(U)\rangle|(\omega)=0
\end{equation}
holds
for all $\phi\in \mathcal{S}_{\varkappa,N}$, $k\in \mathbb{Z}^2$ and $m\in\{0,1\}$.
\end{lemma}
Lemma \ref{MHDlemma0} and Lemma \ref{MHDlemmaB} follow by arguments analogous to \cite[Lemmas 4.6, 4.7]{Huang-2025}; see therein for full proofs.

\begin{lemma}\label{MHDlemmaC} For a certain $k\in\mathbb{Z}_+^2$ and for $\omega\in\Omega_0$,
the following implication 
\begin{equation}
\begin{aligned}
&\sup_{t\in[\eta/2,\eta]}
|\langle K_{t,\eta}\phi,Y_{k}^{m}(U)\rangle|(\omega)
=0\\
&\Rightarrow 
\sup\limits_{l\in\mathcal{Z}_0}
\sup\limits_{t\in[\eta/2,\eta]}|\alpha_l^{m'}|\cdot
|\langle K_{t,\eta}\phi,[Y_{k}^{m}(U),\sigma_{l}^{m'}]
\rangle|(\omega)
=0,
\end{aligned}\end{equation}
holds for all $\phi\in \mathcal{S}_{\varkappa,N},m,m'\in\{0,1\}$. 
\end{lemma}
\begin{proof} 
From Lemma \ref{MHDlemmaB}, it is not difficult for deriving that on $\Omega_0$, one has
\begin{equation}\label{MHDK.4}
\sup\limits_{t\in[\eta/2,\eta]}\langle K_{t,\eta}\phi,Y_k^m(U)\rangle=0,~k\in\mathcal{Z}_0.
\end{equation}
By expanding $U=\bar{U}+Q_b W_{S_t}$, we get that
\begin{equation}
Y_k^m(U)=Y_k^m(\bar{U})-\sum_{l\in\mathcal{Z}_0,m'\in\{0,1\}}
\alpha_l^{m'}[Y_k^m(U),\sigma_l^{m'}]W_{S_t}^{l,m'}.
\end{equation}
We substitute the above formula into \eqref{MHDK.4} to obtain
\begin{equation}
\sup_{t\in[\eta/2,\eta]}\langle K_{t,\eta}\phi,Y_k^m(\bar{U})\rangle
-\sup_{t\in[\eta/2,\eta]}\sum_{l\in\mathcal{Z}_0,m'\in\{0,1\}}
\alpha_l^{m'}\langle K_{t,\eta}\phi,[Y_k^m(U),\sigma_l^{m'}]\rangle W_{S_t}^{l,m'}=0,
\end{equation}
since $\langle K_{t,\eta}\phi,Y_k^m(\bar{U})\rangle$ and
$\langle K_{t,\eta}\phi,[Y_k^m(U),\sigma_l^{m'}]\rangle $ both are nonnegative continuous functions on $[\eta/2,\eta]$,
by Lemma \ref{MHDlemma3.2}, we can derive the desired result. This proof is completed.
\end{proof}

\begin{lemma}\label{MHDlemmaD} For any $k\in\mathcal{Z}_{2n},~n\in\mathbb{N}$,
the following implication
\begin{equation}\label{MHDK.05}
\begin{aligned}
&\sup_{l\in\mathcal{Z}_{0},m,m^{\prime}\in\{0,1\}}
\sup_{t\in[\eta/2,\eta]}|\langle
K_{t,\eta}\phi,
[Y_{k}^{m}(U),\sigma_{l}^{m^{\prime}}]\rangle|(\omega)
=0\\
&\Rightarrow
\sup_{l\in\mathcal{Z}_{0},\ell\notin\{k,-k\}}
\sup_{m\in\{0,1\}}\sup_{t\in[\eta/2,\eta]}|
\langle K_{t,\eta}\phi,\psi_{k+l}^{m}]\rangle|(\omega)
=0
\end{aligned}
\end{equation}
holds for every $\phi\in \mathcal{S}_{\varkappa,N}$ with probability one.
\end{lemma}
\begin{proof}
It directly follows from \eqref{MHD4.1} and Lemma \ref{MHDlemma4.3}.
\end{proof}

\begin{lemma}\label{MHDlemmaE}
For $\phi\in \mathcal{S}_{\varkappa,N}$, $k\in \mathbb{Z}^2$, $m\in\{0,1\}$ and $\omega\in\Omega_0$, it holds that
\begin{equation}\label{MHDK.7}
\begin{aligned}
\sup_{t\in[\eta/2,\eta]}|
\langle K_{t,\eta}\phi,\psi_{k}^{m}\rangle|
=0
\Rightarrow 
\sup_{t\in[\eta/2,\eta]}|\langle K_{t,\eta}\phi,
[F(U),\psi_{k}^{m}]\rangle|=0.
\end{aligned}\end{equation}
\end{lemma}
\begin{proof}
Define $g_{\phi}(t):=\langle K_{t,\eta}\phi,\psi_k^m\rangle$, then
\begin{align*}
g_{\phi}'(t)
&=\langle K_{t,\eta}\phi,[F(U),\psi_k^m]\rangle
=\langle K_{t,\eta}\phi,\mathcal{Y}_k^m(U)\rangle,
\end{align*}
since $g_{\phi}(t):=\langle K_{t,\eta}\phi,\psi_k^m\rangle$ is a nonnegative continuous function on $[\eta/2,\eta]$, we can get that there exists a set $\Omega_0$, and $g_{\phi}'(t)=\langle K_{t,\eta}\phi,\mathcal{Y}_k^m(U)\rangle=0$ holds on the set $\Omega_0$. 
The proof is completed.
\end{proof}

\begin{lemma}\label{MHDlemmaF} 
For all $k\in\mathbb{Z}_+^2$, $m\in\{0,1\}$ and
for all $\phi\in \mathcal{S}_{\varkappa,N}$ and $\omega\in\Omega_0$,
on the set $\Omega_0$, it holds that
\begin{equation*}
\begin{aligned}
\sup_{t\in[\eta/2,\eta]}|\langle
K_{t,\eta}\phi,\mathcal{Y}_{k}^{m}(U)\rangle|(\omega)
=0\Rightarrow
\sup_{l\in\mathcal{Z}_{0},m^{\prime}\in\{0,1\}}
\sup_{t\in[\eta/2,\eta]}|\alpha_{l}^{m^{\prime}}|\cdot
\left|\langle K_{t,\eta}\phi,
[\mathcal{Y}_{k}^{m}(U),\sigma_{l}^{m^{\prime}}])\right|(\omega)
=0.
\end{aligned}
\end{equation*}
\end{lemma}
\begin{proof}
By expanding one has
\begin{equation}
\langle K_{t,\eta}\phi,\mathcal{Y}_{k}^{m}(U)\rangle 
 = \langle K_{t,\eta}\phi,\mathcal{Y}_{k}^{m}(\bar{U})\rangle- \sum_{l\in\mathcal{Z}_{0},m^{\prime}\in\{0,1\}}
 \alpha_{l}^{m^{\prime}}\langle
K_{t,\eta}\phi,[\mathcal{Y}_{k}^{m}(U),
\sigma_{l}^{m^{\prime}}]\rangle W_{S_t}^{l,m'}.
\end{equation}
By the similar argument in Lemma \ref{MHDlemmaC}, we derive the result holds on the set $\Omega_0$.
\end{proof}

\begin{lemma}\label{MHDlemmaF1}
For $\phi\in \mathcal{S}_{\varkappa,N}$, and for any $k\in \mathcal{Z}_{2n+1}, n\in\mathbb{N}$, the following 
\begin{equation}
\begin{aligned}
&\sup_{l\in\mathcal{Z}_{0},m,m^{\prime}\in\{0,1\}}
\sup_{t\in[\eta/2,\eta]}|\langle K_{t,\eta}\phi,
[\mathcal{Y}_{k}^{m}(U),\sigma_{l}^{m^{\prime}}]\rangle|
=0\\
&\Rightarrow
\sup_{l\in\mathcal{Z}_{0},l\notin\{k,-k\}}
\sup_{m\in\{0,1\}}\sup_{t\in[\eta/2,\eta]}|
\langle K_{t,\eta}\phi,\sigma_{k+l}^{m}]\rangle|=0
\end{aligned}\end{equation}
holds with probability one.
\end{lemma}
\begin{proof}
It directly follows from Lemma \ref{MHDlemma4.5}.
\end{proof}

\begin{lemma}\label{MHDlemmaG}
For all $\phi\in \mathcal{S}_{\varkappa,N}$, and for any $ n\in\mathbb{N}$ and $\omega\in\Omega_0$, on the set $\Omega_0$, it holds that
\begin{equation}\label{MHDK.5}
\begin{aligned}
\sum_{k\in \mathcal{Z}_{2n},m\in\{0,1\}}\sup_{t\in[\eta/2,\eta]}|
\langle K_{t,\eta}\phi,\sigma_{k}^{m}\rangle|
=0
\Rightarrow
\sum_{k\in \mathcal{Z}_{2n},m\in\{0,1\}}
\sup_{t\in[\eta/2,\eta]}|\langle K_{t,\eta}\phi,
\psi_{k}^{m}\rangle|=0.
\end{aligned}\end{equation}
\end{lemma}
\begin{proof}
By Lemma \ref{MHDlemmaB}, for any $k\in\mathcal{Z}_{2n},m\in\{0,1\}, \omega\in \Omega_0$, there holds
\begin{equation}
\sup_{t\in[\eta/2,\eta]}|\langle
K_{t,\eta}\phi,\sigma_k^m\rangle|(\omega)=0~~
\Rightarrow~~\sup_{t\in[\eta/2,\eta]}|\langle
K_{t,\eta}\phi,Y_k^m(U)\rangle|(\omega)=0.
\end{equation}
Then by Lemma \ref{MHDlemmaC}, for any $k\in\mathbb{Z}_+^2, \omega\in \Omega_0, m,m'\in\{0,1\}$, one knows
\begin{equation}
\begin{aligned}
&\sup_{t\in[\eta/2,\eta]}
|\langle K_{t,\eta}\phi,Y_{k}^{m}(U)\rangle|(\omega)
=0\\
&\Rightarrow
\sup\limits_{l\in\mathcal{Z}_0}
\sup\limits_{t\in[\eta/2,\eta]}|\alpha_l^{m'}|\cdot
|\langle K_{t,\eta}\phi,[Y_{k}^{m}(U),\sigma_{l}^{m'}]
\rangle|(\omega)
=0.
\end{aligned}\end{equation}
Next, by Lemma \ref{MHDlemmaD}, for any $k\in\mathcal{Z}_{2n},n\in\mathbb{N}$, one has
\begin{equation}
\begin{aligned}
&\sup_{l\in\mathcal{Z}_{0},m,m^{\prime}\in\{0,1\}}
\sup_{t\in[\eta/2,\eta]}|\langle
K_{t,\eta}\phi,
[Y_{k}^{m}(U),\sigma_{l}^{m^{\prime}}]\rangle|(\omega)
=0\\
&\Rightarrow
\sup_{l\in\mathcal{Z}_{0},\ell\notin\{k,-k\}}
\sup_{m\in\{0,1\}}\sup_{t\in[\eta/2,\eta]}|
\langle K_{t,\eta}\phi,\psi_{k+l}^{m}]\rangle|(\omega)
=0.
\end{aligned}
\end{equation}
In summary, \eqref{MHDK.5} holds on the set $\Omega_0$. This proof is completed.
\end{proof}

\begin{lemma}\label{MHDlemmaH}
For $\phi\in \mathcal{S}_{\varkappa,N}$, $k\in \mathbb{Z}^2$, $m\in\{0,1\}$ and $\omega\in\Omega_0$, it holds that
\begin{equation}\label{MHDK.6}
\begin{aligned}
\sup_{t\in[\eta/2,\eta]}|
\langle K_{t,\eta}\phi,\psi_{k}^{m}\rangle|
=0
\Rightarrow
\sup_{t\in[\eta/2,\eta]}|\langle K_{t,\eta}\phi,
\sigma_{k}^{m}\rangle|=0.
\end{aligned}\end{equation}
\end{lemma}
\begin{proof}
By Lemma \ref{MHDlemmaE}, for $m\in\{0,1\},k\in\mathcal{Z}_{2n+1}$, we get on the set $\Omega_0$
\begin{equation}
\begin{aligned}
\sup_{t\in[\eta/2,\eta]}|
\langle K_{t,\eta}\phi,\psi_{k}^{m}\rangle|(\omega)
=0
\Rightarrow 
\sup_{t\in[\eta/2,\eta]}|\langle K_{t,\eta}\phi,
\mathcal{Y}_{k}^{m}\rangle|(\omega)=0.
\end{aligned}\end{equation}
By Lemma \ref{MHDlemmaF}, for $m\in\{0,1\},k\in\mathcal{Z}_{2n+1}$, on the set $\Omega_0$, we have
\begin{equation*}
\begin{aligned}
\sup_{t\in[\eta/2,\eta]}|\langle
K_{t,\eta}\phi,\mathcal{Y}_{k}^{m}(U)\rangle|(\omega)
=0\Rightarrow
\sup_{l\in\mathcal{Z}_{0},m^{\prime}\in\{0,1\}}
\sup_{t\in[\eta/2,\eta]}|\alpha_{l}^{m^{\prime}}|\cdot
\left|\langle K_{t,\eta}\phi,
[\mathcal{Y}_{k}^{m}(U),\sigma_{l}^{m^{\prime}}])\right|(\omega)
=0.
\end{aligned}
\end{equation*}
Finally, by Lemma \ref{MHDlemmaG}, we obtain that \eqref{MHDK.6} holds on the set $\Omega_0$. This proof is completed.
\end{proof}


To facilitate later applications, we will prove a stronger result than  \eqref{MHD3.4}:
\begin{equation}\label{MHD3.5}
\mathbb{P}\Big(\omega=(\mathrm w,\ell):\inf\limits_{\phi\in\mathcal{S}_{\varkappa,N}}
\sum\limits_{j\in \mathbb{Z}_+^2,m\in\{0,1\}}\int_{\eta/2}^{\eta}\left(\left|\langle K_{r,\eta}\phi,\sigma_j^m\rangle\right|^2
+\left|\langle K_{r,\eta}\phi,\psi_j^m\rangle\right|^2\right)
\mathrm d\ell_r=0\Big)=0.
\end{equation}

\begin{proof}[Proof of Proposition \ref{MHDpropo3.4}] We will use Lemma \ref{MHDlemmaG} and Lemma \ref{MHDlemmaF} to prove \eqref{MHD3.5}. 
%
Define the event set:
\begin{equation}\label{MHD3.6}
\mathcal{L}:=
\Big\{\omega:\inf\limits_{\phi\in\mathcal{S}_{\varkappa,N}}
\sum\limits_{j\in \mathcal{Z}_0,m\in\{0,1\}}\int_{\eta/2}^{\eta}
\langle K_{r,\eta}\phi,\sigma_j^m\rangle^2
\mathrm d\ell_r=0\Big\}
\cap\Omega_{0}\cap\tilde{\Omega}_{0}.
\end{equation}
Assume that $\mathcal{L}\neq\emptyset$ and let $\omega=(\mathrm w,\ell)$ belong to the event $\mathcal{L}$. Then there exists some $\phi$ with
\begin{equation}\label{MHD3.7}
\|P_{N}\phi\|\geq\varkappa,
\end{equation}
such that on the set $\Omega_{0}\cap\tilde{\Omega}_{0}$, for all $j\in\mathcal Z_0$ and $m\in\{0,1\}$, we have:
\begin{equation*}
\begin{aligned}
\int_{\eta/2}^{\eta}\langle K_{r,\eta}\phi,\sigma_j^m\rangle^2
\mathrm{d}\ell_{r}=0
~\text{and}~
\int_{\eta/2}^{\eta}\langle K_{r,\eta}\phi,\psi_j^m
\rangle^2 \mathrm{d}\ell_{r}=0.
\end{aligned}
\end{equation*}
Furthermore, by the properties of $\ell\in\Omega_0\cap\tilde{\Omega}_0$ stated above and the continuity of $\langle K_{t,\eta}\phi,\sigma_{j}^{m}\rangle$ and $\langle K_{t,\eta}\phi,\psi_j^m\rangle$ with respect to $t$, we obtain
\begin{equation*}
\begin{aligned}
\begin{cases}
&\sup\limits_{t\in[\eta/2,\eta]}|\langle K_{t,\eta}\phi,\sigma_{j}^{m}\rangle|=0,\\
&\sup\limits_{t\in[\eta/2,\eta]}
|\langle K_{t,\eta}\phi,\psi_j^m
\rangle|=0,~\forall j\in\mathcal{Z}_0,~m\in\{0,1\}.
\end{cases}
\end{aligned}
\end{equation*}
By Lemma \ref{MHDlemmaF}, there exists $\phi\in \mathcal{S}_{\varkappa,N}$ such that for all $j\in \mathbb{Z}_+^2$ and $m\in\{0,1\}$, the above suprema vanish. Taking $t\rightarrow \eta$, we conclude that $\phi=0$, which contradicts $\|P_{N}\phi\| \geq \varkappa$. Therefore, the Proposition \ref{MHDpropo3.4} is proved.
\end{proof} 

Building upon the dissipative property of the MHD system and the estimate \eqref{MHD3.5}, we aim to establish the following convergence:
\begin{equation}\label{MHDr}
r(\varepsilon,\varkappa,\mathfrak{R},N)\rightarrow0,
~~\text{as}~~ \varepsilon\rightarrow 0.
\end{equation}

We proceed by contradiction. Assume that \eqref{MHDr} fails to hold. Then there exist the sequences $\{U_0^{(k)}\}\subseteq B_H(\mathfrak{R})$, $\{\varepsilon_k\}\subseteq(0,1)$ and a positive number $\delta_0$ satisfying both
\begin{equation}\label{MHDB.1}
\lim\limits_{k\to\infty}\mathbb{P}(X^{U_0^{(k)},\varkappa,N}
<\varepsilon_k)\ge\delta_0>0
~~\text{and}~\lim\limits_{k\to\infty}\varepsilon_k=0.
\end{equation}
Thus we need to search out something to contradict \eqref{MHDB.1}, then \eqref{MHDr} can be proved.

By the weak compactness of the Hilbert space $H$, we may extract a subsequence (still denoted by $\{U_0^{(k)},k\geq 1\}$ for simplicity) that converges weakly to some $U_0^{(0)}\in H$. Let $U_t^{(k)}$ denote the solution to equation \eqref{MHD2.14} with initial condition $\{U_0^{(k)},k\geq 1\}$, governed by the linearized equation
\begin{equation}\label{MHDB.2}
\partial_t J_{s,t}\xi+A^{\alpha,\beta}J_{s,t}\xi+\nabla B(U_t)J_{s,t}\xi=0,\quad J_{s,s}\xi=\xi,
\end{equation}
when $U_t$ is replaced by $U_t^{(k)}$, where $\nabla B(U_t)J_{s,t}\xi:=B(U_t,J_{s,t}\xi)+B(J_{s,t}\xi,U_t)$. We denote by $J_{s,t}^{(k)}\xi$ and $K_{s,t}^{(k)}\xi$ the corresponding solution of \eqref{MHDB.2} and its adjoint operator, respectively.

For any $N \geq 1$, we introduce the finite-dimensional subspaces
\begin{equation*}
H_N:=\mathrm{span}\left\{\sigma_k^l,\psi_k^l:|k|\leq N, l\in\{0,1\}\right\}
\end{equation*}
with orthogonal projections
$P_N:H\to H_N $ and $ Q_N:=I-P_N$. 
Note that $Q_N$ projects $H$ onto the high-frequency components $\text{span}\left\{\sigma_k^l,\psi_k^l:|k|>N,l\in\{0,1\}\right\}$.


The proof strategy consists of three main steps: establish estimates for $\|Q_{M}U_{t}^{(k)}\|$ (high-frequency control)); analyze $\|P_{M}U_{t}^{(k)}-P_{M}U_{t}^{(0)}\|$ (finite-dimensional approximation); study $\|J_{s,t}^{(k)}\xi-J_{s,t}^{(0)}\xi\|,s,t\in(0,T]$ (linearized equation stability). These estimates, combined with the methodology from \cite[Appendix B]{Peng-2024} , will ultimately lead to the desired contradiction of \eqref{MHDB.1}, thereby proving \eqref{MHDr}.

\begin{lemma}\label{MHDlemma 4.4} For any $t\geq0,k\in\mathbb{N}$ and $M>\max\{|k|:k\in\mathcal{Z}_0\}$, there exist two positive constant $C$ and $C_{\mathfrak{R}}$, they depend on $\nu_S,d$ and  $\mathfrak{R},\nu_S,d$, respectively, such that
 \begin{equation}\label{MHD4.5}
\begin{aligned}
\|Q_{M}U_{t}^{(k)}\|^{2}\leq& e^{-M^{2}t}\|Q_{M}U_{0}^{(k)}\|^{2}
+\frac{C}{M^{2}}(1-e^{-M^{2}t})\int_0^t \|U_{s}^{(k)}\|_1^2\mathrm{d}s \sup\limits_{s\in[0,t]}\|U_{s}^{(k)}\|^2,
\end{aligned}
\end{equation}
and
\begin{equation}\label{MHD4.6}
\begin{aligned}
&\|P_{M}U_{t}^{(k)}-P_M U_t^{(0)}\|^{2}\\
\leq& \|P_{M}U_{0}^{(k)}-P_M U_0^{(0)}\|^{2}
e^{C\int_0^t(\|U_{s}^{(k)}\|_1^2+\|U_{s}^{(0)}\|_1^2) \mathrm{d}s}\\
&+Ce^{C\int_0^t\|U_{s}^{(k)}\|_1^2+\|U_{s}^{(0)}\|_1^2 \mathrm{d}s}\int_0^t \|Q_{M}U_{s}^{(k)}-Q_M U_s^{(0)}\|^2
(\|U_{s}^{(k)}\|_1^2+\|U_{s}^{(0)}\|_1^2)\mathrm{d}s.
\end{aligned}\end{equation}
\end{lemma}
\begin{proof} First, for the nonlinear terms, we deduce from the interpolation inequality that
\begin{align*}
&\,\, |-\langle \tilde{B}(u_{t}^{(k)},u_{t}^{(k)}),Q_{M}u_{t}^{(k)}\rangle
+\langle \tilde{B}(b_{t}^{(k)},b_{t}^{(k)}),Q_{M}u_{t}^{(k)}\rangle|\\
&+|-\langle \tilde{B}(u_{t}^{(k)},b_{t}^{(k)}),Q_{M}b_{t}^{(k)}\rangle
+\langle \tilde{B}(b_{t}^{(k)},u_{t}^{(k)}),Q_{M}b_{t}^{(k)}\rangle| \\
&\leq C\left[\|\Lambda^{\alpha}Q_{M}u_{t}^{(k)}\|
+\|\Lambda^{\beta}Q_{M}b_{t}^{(k)}\|\right]
\left[\|u_{t}^{(k)}\|\|u_{t}^{(k)}\|_1\right]\\
&~~~+C\left[\|\Lambda^{\alpha}Q_{M}u_{t}^{(k)}\|
+\|\Lambda^{\beta}Q_{M}b_{t}^{(k)}\|\right]
\left[\|b_{t}^{(k)}\|\|b_{t}^{(k)}\|_1\right]\\
&~~~+C\left[\|\Lambda^{\alpha}Q_{M}u_{t}^{(k)}\|
+\|\Lambda^{\beta}Q_{M}b_{t}^{(k)}\|\right]
\left[\|u_{t}^{(k)}\|\|b_{t}^{(k)}\|_1\right]\\
&~~~+C\left[\|\Lambda^{\alpha}Q_{M}u_{t}^{(k)}\|
+\|\Lambda^{\beta}Q_{M}b_{t}^{(k)}\|\right]
\left[\|b_{t}^{(k)}\|\|u_{t}^{(k)}\|_1\right]\\
&\leq \frac{1}{2}\left[\nu_1\|\Lambda^{\alpha}Q_{M}u_{t}^{(k)}\|^2
\nu_2\|\Lambda^{\beta}Q_{M}b_{t}^{(k)}\|^2\right]
+C\|u_{t}^{(k)}\|^2\|u_{t}^{(k)}\|_1^2
+C\|b_{t}^{(k)}\|^2\|b_{t}^{(k)}\|_1^2\\
&~~~+C\|u_{t}^{(k)}\|^2\|b_{t}^{(k)}\|_1^2
+C\|b_{t}^{(k)}\|^2\|u_{t}^{(k)}\|_1^2
.
\end{align*}
We take the inner product of \eqref{MHD1.1} with $Q_{M}u_{t}^{(k)}$ and $Q_{M}b_{t}^{(k)}$, respectively, there holds
\begin{equation}\label{MHD4.10}
\begin{aligned}
&\frac{\mathrm{d}}{\mathrm{d}t}(\|Q_{M}u_{t}^{(k)}\|^{2}+\|Q_{M}b_{t}^{(k)}\|^{2})
+\|\nu_1\Lambda^{\alpha}Q_{M}u_{t}^{(k)}\|^2+\|\nu_2\Lambda^{\beta}Q_{M}b_{t}^{(k)}\|^2\\
&\leq C(\|u_{t}^{(k)}\|_{1}^2+\|b_{t}^{(k)}\|^{2})
(\|u_{t}^{(k)}\|^2+\|b_{t}^{(k)}\|^{2}).
\end{aligned}
\end{equation}
The above inequality follows that
\begin{align*}
&\|Q_{M}u_{t}^{(k)}\|^{2}+\|Q_{M}b_{t}^{(k)}\|^{2}\\
\leq& e^{-\nu M^{2}t}\left(\|Q_{M}u_{0}^{(k)}\|^{2}
+\|Q_{M}b_{0}^{(k)}\|^{2}\right)
+C\int_0^t e^{-\nu M^{2}(t-s)}\|U_{t}^{(k)}\|_1^2\|U_{t}^{(k)}\|^2 \mathrm{d}s\\
\leq& e^{-\nu M^{2}t}\|Q_{M}U_{0}^{(k)}\|^{2}
+\frac{C}{\nu M^{2}}(1-e^{-\nu M^{2}t})\int_0^t \|U_{s}^{(k)}\|_1^2\mathrm{d}s \sup\limits_{s\in[0,t]}\|U_{s}^{(k)}\|^2.
\end{align*}

Next, we turn to prove \eqref{MHD4.6}. One has
\begin{align}\label{MHD4.7}
 & \frac{\mathrm{d}}{\mathrm{d}t}(\|P_{M}u_{t}^{(k)}-P_{M}u_{t}^{(0)}\|^{2}+\|P_{M}b_{t}^{(k)}-P_{M}b_{t}^{(0)}\|^{2}) \notag \\
 =&-2\nu_1\langle(-\Delta)^{\alpha}(P_{M}u_{t}^{(k)}
 -P_{M}u_{t}^{(0)}),P_{M}u_{t}^{(k)}-P_{M}u_{t}^{(0)}\rangle \notag \\
 &  -2\nu_2\langle(-\Delta)^{\beta}
 (P_{M}b_{t}^{(k)}-P_{M}b_{t}^{(0)}),
 P_{M}b_{t}^{(k)}-P_{M}b_{t}^{(0)}\rangle \notag\\
 &  -2\langle B(u_{t}^{(k)},u_{t}^{(k)}),B(u_{t}^{(0)},u_{t}^{(0)}),
 P_{M}u_{t}^{(k)}-P_{M}u_{t}^{(0)}\rangle \\
 &  +2\langle B(b_{t}^{(k)},b_{t}^{(k)})-B(b_{t}^{(0)},b_{t}^{(10)}),
 P_{M}u_{t}^{(k)}-P_{M}u_{t}^{(0)}\rangle \notag\\
 &  -2\langle B(u_{t}^{(k)},b_{t}^{(k)})-B(u_{t}^{(0)},b_{t}^{(0)}),
 P_{M}b_{t}^{(k)}-P_{M}b_{t}^{(0)})\rangle \notag \\
 &  +2\langle B(b_{t}^{(k)},u_{t}^{(k)})-B(b_{t}^{(0)},u_{t}^{(0)}),
 P_{M}b_{t}^{(k)}-P_{M}b_{t}^{(0)})\rangle \notag \\
 &=:I_i,~i=1,\cdots,6. \notag
\end{align}
We now estimate $I_3-I_6$ in turn, 
\begin{align*}
  I_3 =& -2\langle \tilde{B}(u_{t}^{(k)},Q_M (u_{t}^{(k)}-u_{t}^{(0)})),P_{M}u_{t}^{(k)}-P_{M}u_{t}^{(0)}\rangle 
  -2\langle \tilde{B}(u_{t}^{(k)}-u_{t}^{(0)},u_{t}^{(0)}),P_{M}u_{t}^{(k)}-P_{M}u_{t}^{(0)}\rangle \\
  \leq & C\|u_{t}^{(k)}\|_1^2\|Q_M (u_{t}^{(k)}-u_{t}^{(0)})\|^2
  +C\|P_{M}(u_{t}^{(k)}-u_{t}^{(0)})\|_1^2 \\
   & +C\|P_{M}(u_{t}^{(k)}-u_{t}^{(0)})\|^2\|u_{t}^{(0)}\|_1^2
   +C\|Q_M (u_{t}^{(k)}-u_{t}^{(0)})\|^2\|u_{t}^{(0)}\|_1^2
   +C\|P_{M}(u_{t}^{(k)}-u_{t}^{(0)})\|_1^2\\
   \leq & C\|u_{t}^{(k)}\|_1^2\|Q_M (u_{t}^{(k)}-u_{t}^{(0)})\|^2
   +C\|P_{M}(u_{t}^{(k)}-u_{t}^{(0)})\|^2\|u_{t}^{(0)}\|_1^2
   +C\|Q_M (u_{t}^{(k)}-u_{t}^{(0)})\|^2\|u_{t}^{(0)}\|_1^2\\
   &+\frac{\nu_1}{2}\|\Lambda^\alpha P_{M}(u_{t}^{(k)}-u_{t}^{(0)})\|^2,
\end{align*}
and
\begin{align*}
  I_4 =& 2\langle \tilde{B}(b_{t}^{(k)},b_{t}^{(k)}-b_{t}^{(0)}),P_{M}(u_{t}^{(k)}-u_{t}^{(0)})\rangle 
  +2\langle \tilde{B}(b_{t}^{(k)}-b_{t}^{(0)},b_{t}^{(0)}),P_{M}(u_{t}^{(k)}-u_{t}^{(0)})\rangle \\
  \leq & C(\|b_{t}^{(k)}\|_1^2+\|b_{t}^{(0)}\|_1^2)\|P_M (b_{t}^{(k)}-b_{t}^{(0)})\|^2
  +C\|P_{M}(u_{t}^{(k)}-u_{t}^{(0)})\|_1^2 
  \\
  &+C(\|b_{t}^{(k)}\|_1^2+\|b_{t}^{(0)}\|_1^2)\|Q_M (b_{t}^{(k)}-b_{t}^{(0)})\|^2\\
  \leq & C(\|b_{t}^{(k)}\|_1^2+\|b_{t}^{(0)}\|_1^2)\|P_M (b_{t}^{(k)}-b_{t}^{(0)})\|^2
  +C(\|b_{t}^{(k)}\|_1^2+\|b_{t}^{(0)}\|_1^2)\|Q_M (b_{t}^{(k)}-b_{t}^{(0)})\|^2\\
  &+\frac{\nu_1}{2}\|\Lambda^\alpha (P_{M}(u_{t}^{(k)}-u_{t}^{(0)})\|^2,
\end{align*}
and
\begin{align*}
  I_5 =&-2 \langle \tilde{B}(u_{t}^{(k)},Q_M(b_{t}^{(k)}-b_{t}^{(0)}),P_{M}b_{t}^{(k)}-P_{M}b_{t}^{(0)}\rangle 
  -2\langle \tilde{B}(u_{t}^{(k)}-u_{t}^{(0)},b_{t}^{(0)}),P_{M}b_{t}^{(k)}-P_{M}b_{t}^{(0)}\rangle \\
  \leq & C\|u_{t}^{(k)}\|_1^2\|Q_M (b_{t}^{(k)}-b_{t}^{(0)})\|^2
 +C\|b_{t}^{(0)}\|_1^2\|P_{M}u_{t}^{(k)}-P_{M}u_{t}^{(0)}\|^2 
  +C\|b_{t}^{(0)}\|_1^2\|Q_M (u_{t}^{(k)}-u_{t}^{(0)})\|^2\\
  &+\frac{\nu_2}{2}\|\Lambda^\beta (P_{M}b_{t}^{(k)}-P_{M}b_{t}^{(0)})\|^2,
\end{align*}
and
\begin{align*}
  I_6 =&2 \langle \tilde{B}(b_{t}^{(k)},u_{t}^{(k)}-u_{t}^{(0)})
  +\tilde{B}(b_{t}^{(k)}-b_{t}^{(0)},u_{t}^{(0)}),P_{M}b_{t}^{(k)}-P_{M}b_{t}^{(0)}\rangle 
  \\
  \leq & C\|b_{t}^{(k)}\|_1^2\|P_M (u_{t}^{(k)}-u_{t}^{(0)})\|^2
  +C\|b_{t}^{(k)}\|_1^2\|Q_{M}(u_{t}^{(k)}-u_{t}^{(0)})\|_1^2 \\
&  +C\|P_{M}(b_{t}^{(k)}-b_{t}^{(0)})\|_1^2 
  +C\|u_{t}^{(0)}\|_1^2\|P_M (b_{t}^{(k)}-b_{t}^{(0)})\|^2
  +C\|u_{t}^{(0)}\|_1^2\|Q_M (b_{t}^{(k)}-b_{t}^{(0)})\|^2\\
  \leq & C\|b_{t}^{(k)}\|_1^2\|P_M (u_{t}^{(k)}-u_{t}^{(0)})\|^2
  +C\|b_{t}^{(k)}\|_1^2\|Q_{M}(u_{t}^{(k)}-u_{t}^{(0)})\|_1^2 \\
& 
  +C\|u_{t}^{(0)}\|_1^2\|P_M (b_{t}^{(k)}-b_{t}^{(0)})\|^2
  +C\|u_{t}^{(0)}\|_1^2\|Q_M (b_{t}^{(k)}-b_{t}^{(0)})\|^2\\
  &+\frac{\nu_2}{2}\|\Lambda^\beta P_{M}(b_{t}^{(k)}-b_{t}^{(0)})\|^2.
\end{align*}
By organizing these estimates with \eqref{MHD4.7}, and then using Gronwall inequality, we can obtain
\begin{equation}
\begin{aligned}
&\|P_{M}U_{t}^{(k)}-P_M U_t^{(0)}\|^{2}\\
\leq& \|P_{M}U_{0}^{(k)}-P_M U_0^{(0)}\|^{2}
e^{C\int_0^t(\|U_{s}^{(k)}\|_1^2+\|U_{s}^{(0)}\|_1^2) \mathrm{d}s}\\
&+Ce^{C\int_0^t\|U_{s}^{(k)}\|_1^2+\|U_{s}^{(0)}\|_1^2 \mathrm{d}s}\int_0^t \|Q_{M}U_{s}^{(k)}-Q_M U_s^{(0)}\|^2
(\|U_{s}^{(k)}\|_1^2+\|U_{s}^{(0)}\|_1^2)\mathrm{d}s.
\end{aligned}
\end{equation}
The proof is completed.
\end{proof}

\begin{lemma}\label{MHDlemma4.4} For any $0\leq s\leq t, k\in\mathbb{N}$ and $\xi\in H$ with $\|\xi\|=1$, one has
\begin{equation*}
\begin{aligned}
&\|J_{s,t}^{(k)}\xi-J_{s,t}^{(0)}\xi\|^{2} 
\leq C \sup_{r\in[s,t]}\|U_{r}^{(k)}-U_{r}^{(0)}\|^{2}\cdot 
e^{C\int_{s}^{t}\|U_{r}^{(k)}\|_{1}^{2}\mathrm{d}r} \int_{s}^{t}\|J_{s,r}^{(0)}\xi\|_{1}^{2}
\mathrm{d}r,
\end{aligned}
\end{equation*}
where $C$ is a positive constant depending on $\nu_S,d$.
\end{lemma}
\begin{proof} Denote that $J_{s,t}^{(k)}\xi-J_{s,t}^{(0)}\xi=(J_{s,t}^{1,(k)}\xi-J_{s,t}^{1,(0)}\xi,J_{s,t}^{2,(k)}\xi-J_{s,t}^{2,(0)}\xi)$. By the equation \eqref{MHDB.2}, it is easy to obtain that
\begin{equation}\label{MHDB.7}
\begin{aligned}
&\frac{\mathrm{d}}{\mathrm{d}t}\|J_{s,t}^{(k)}\xi-J_{s,t}^{(0)}\xi\|^{2}
=-2\langle A^{\alpha,\beta}(J_{s,t}^{(k)}\xi-J_{s,t}^{(0)}\xi)
,J_{s,t}^{(k)}\xi-J_{s,t}^{(0)}\xi\rangle \\
&~~~-2\langle J_{s,t}^{(k)}\xi-J_{s,t}^{(0)}\xi,
B(U_{t}^{(k)},J_{s,t}^{(k)}\xi)
-B(U_{t}^{(0)},J_{s,t}^{(0)}\xi)\rangle  \\
&~~~-2\langle J_{s,t}^{(k)}\xi-J_{s,t}^{(0)}\xi,
B(J_{s,t}^{(k)}\xi,U_{t}^{(k)})
-B(J_{s,t}^{(0)}\xi,U_{t}^{(0)})\rangle 
\\
&:=-2\nu_1
\|\Lambda^{\alpha}(J_{s,t}^{1,(k)}\xi-J_{s,t}^{1,(0)}\xi)\|^2
-2\nu_2
\|\Lambda^{\beta}(J_{s,t}^{2,(k)}\xi-J_{s,t}^{2,(0)}\xi)\|^2
+2(P_1+P_2).
\end{aligned}
\end{equation}
For the terms $P_1$ and $P_2$, we have
\begin{align*}
|P_1 |=& \left|\langle J_{s,t}^{(k)}\xi-J_{s,t}^{(0)}\xi,
B(U_{t}^{(k)},J_{s,t}^{(k)}\xi)
-B(U_{t}^{(0)},J_{s,t}^{(0)}\xi)\rangle\right| \\
\leq & C\|U_{t}^{(k)}-U_{t}^{(0)}\|
\left[\|\Lambda^{\alpha}(J_{s,t}^{1,(k)}\xi-J_{s,t}^{1,(0)}\xi)\|
+\|\Lambda^{\beta}(J_{s,t}^{2,(k)}\xi-J_{s,t}^{2,(0)}\xi)\|\right]
\|J_{s,t}^{(0)}\xi\|_1 \\
\leq &  \frac{\nu_1}{4}
\|\Lambda^{\alpha}(J_{s,t}^{1,(k)}\xi-J_{s,t}^{1,(0)}\xi)\|^2
+\frac{\nu_2}{4}
\|\Lambda^{\beta}(J_{s,t}^{2,(k)}\xi-J_{s,t}^{2,(0)}\xi)\|^2 \\
&+C\|U_t^{(k)}-U_t^{(0)}\|^2\|J_{s,t}^{(0)}\xi\|_{1}^2
,
\end{align*}
and
\begin{align*}
|P_{2}|= &\left| \langle J_{s,t}^{(k)}\xi-J_{s,t}^{(0)}\xi,B(J_{s,t}^{(k)}\xi-J_{s,t}^{(0)}\xi,U_{t}^{(k)})
+B(J_{s,t}^{(0)}\xi,U_{t}^{(k)}-U_{t}^{(0)})\rangle\right| \\
 \leq& C\|U_t^{(k)}-U_t^{(0)}\|
\left[\|\Lambda^{\alpha}(J_{s,t}^{1,(k)}\xi-J_{s,t}^{1,(0)}\xi)\|
+\|\Lambda^{\beta}(J_{s,t}^{2,(k)}\xi-J_{s,t}^{2,(0)}\xi)\|\right]
\|J_{s,t}^{(0)}\xi\|_{1}\\
&+ C\|U_{t}^{(k)}\|_{1}\|J_{s,t}^{(k)}\xi-J_{s,t}^{(0)}\xi\|
\left[\|\Lambda^{\alpha}(J_{s,t}^{1,(k)}\xi-J_{s,t}^{1,(0)}\xi)\|
+\|\Lambda^{\beta}(J_{s,t}^{2,(k)}\xi-J_{s,t}^{2,(0)}\xi)\|\right]
\\
\leq&\frac{\nu_1}{4}
\|\Lambda^{\alpha}(J_{s,t}^{1,(k)}\xi-J_{s,t}^{1,(0)}\xi)\|^2
+\frac{\nu_2}{4}
\|\Lambda^{\beta}(J_{s,t}^{2,(k)}\xi-J_{s,t}^{2,(0)}\xi)\|^2 \\
 &+C\|U_t^{(k)}\|_1^2\|J_{s,t}^{(k)}\xi-J_{s,t}^{(0)}\xi\|^2+C\|U_{t}^{(k)}-U_{t}^{(0)}\|^{2}\|J_{s,t}^{(0)}\xi\|_{1}^{2}.
\end{align*}
Combining \eqref{MHDB.7} with the estimates of $P_1$ and $P_2$, we obtain
\begin{equation*}
\begin{aligned}
&\|J_{s,t}^{(k)}\xi-J_{s,t}^{(0)}\xi\|^{2} 
\leq C \sup_{r\in[s,t]}\|U_{r}^{(k)}-U_{r}^{(0)}\|^{2}\cdot 
e^{C\int_{s}^{t}\|U_{r}^{(k)}\|_{1}^{2}\mathrm{d}r} \int_{s}^{t}\|J_{s,r}^{(0)}\xi\|_{1}^{2}
\mathrm{d}r.
\end{aligned}
\end{equation*}
The proof is completed.
\end{proof}

By synthesizing Lemma  \ref{MHDlemma 4.4} and Lemma \ref{MHDlemma4.4} with the technical framework developed in \cite[Applendix B]{Peng-2024}, we derive the following result:
\begin{proposition}\label{MHDproper3.5}
For $\varkappa\in(0,1],\mathfrak{R}>0$ and $N\in \mathbb{N}$, we have
\begin{equation*}
\lim_{\varepsilon\rightarrow 0}r(\varepsilon,\varkappa,\mathfrak{R},N)
=0.
\end{equation*}
\end{proposition}

\section{Weak irreducibility and ergodicity}
The ergodicity of invariant measures can usually be obtained by proving the irreducibility and e-property of Markov process, or strong Feller property, or asymptotic strong Feller property, refer to \cite{prato-1996,MH-2006,Kapica-2012,Komorowski-2010}. In this section, we will establish e-property and weak irreducibility to obtain the ergodicity for the system \eqref{MHD1.1}.

\begin{proof}[Proof of Proposition \ref{MHDpropo1.4}]
Once Proposition \ref{MHDpropo3.4} is established and Proposition \ref{MHDproper3.5} is guaranteed, it becomes possible to convert spectral bounds on the Malliavin matrix $\mathcal M$ into estimates for $\nabla P_t \Phi$. Since the work presented in \cite{MH-2006}, significant attention has been devoted to the proof of this type of gradient inequality about $\|\nabla P_t \Phi(U_0)\|$, with continuous refinements and advancements being made over time. The recent improvement contribution comes from the article \cite{Peng-2024}, they utilized the stopping time technique to process and improve \textbf{the choice of $v$} and \textbf{the control of $\rho_{\eta_n}$}, respectively, and accordingly, new moment estimates will also be provided regarding \textbf{the control of $\int_0^{\ell_t}v(s)\mathrm{d}W(s)$}. Broadly speaking, supplied with moment estimates of $U,J_{s,t}\xi,J_{s,t}^{(2)}(\xi,\xi'),\mathcal M_{s,t}$ listed in Section 3, one need to formulate a control problem through the Malliavin integration by parts formula, then do some decay estimates adopting an iterative construction with the aid of Lemma \ref{MHDlemma 2.7}, Lemma \ref{MHDlemma2.10}, Lemma \ref{MHDlemma2.1}, Lemma \ref{MHDlemma4.5}. We refer the readers to \cite{Peng-2024} and omit the details.
\end{proof}

%
%
%
%

\begin{proof}[Proof of Proposition \ref{MHDirre}]
Define $(V_t^{(1)},V_t^{(2)})^T:=V_t=U_t-\zeta_t=:(u_t-\zeta_t^{(1)},b_t-\zeta_t^{(2)})^T$, where $\zeta_t=\sum_{k\in \mathcal{Z}_0,l\in\{0,1\}}\alpha_k^l\sigma_k^l W_{S_t}^{k,l}$, the process $V_t$ satisfies  
\begin{equation*}
\frac{\partial V_t}{\partial t}
=- A^{\alpha,\beta}(V_t+\zeta_t)-B(U_t,U_t)
=- A^{\alpha,\beta}(V_t+\zeta_t)-B(U_t,V_t+\zeta_t).
\end{equation*} 
Applying the inner product operation with 
$V_t$ to this equation, we arrive at
\begin{equation*}
\begin{aligned}
\frac {\mathrm{d}}{\mathrm{d}t}\|V_t\|^2 &=-2\nu_1\|\Lambda^\alpha V^{(1)}_{t}\|^{2}
-2\nu_2\|\Lambda^\beta V^{(2)}_{t}\|^{2}
+2\langle A^{\alpha,\beta}\zeta_{t},V_{t}\rangle
+2\langle B(U_{t},\zeta_{t}),V_{t}\rangle  \\
&\leq -2\nu_1\|\Lambda^\alpha V^{(1)}_{t}\|^{2}
-2\nu_2\|\Lambda^\beta V^{(2)}_{t}\|^{2}
+2\|\Lambda^\alpha \zeta_t^{(1)}\|
\|\Lambda^\alpha V_t^{(1)}\|
+2\|\Lambda^\beta \zeta_t^{(2)}\|
\|\Lambda^\alpha V_t^{(2)}\| 
+C_1\|U_t\|\|\zeta_t\|_1\|V_t\|_1 \\
&\leq -\frac{3\nu_1}{2}\|\Lambda^\alpha V^{(1)}_{t}\|^{2}
-\frac{3\nu_2}{2}\|\Lambda^\beta V^{(2)}_{t}\|^{2}
+2\nu_1\|\Lambda^\alpha \zeta_t^{(1)}\|^2
+2\nu_2\|\Lambda^\beta \zeta_t^{(2)}\|
+C_1\|V_t+\zeta_t\|\|\zeta_t\|_1\|V_t\|_1\\
&\leq -\frac{3\nu_1}{2}\|\Lambda^\alpha V^{(1)}_{t}\|^{2}
-\frac{3\nu_2}{2}\|\Lambda^\beta V^{(2)}_{t}\|^{2}
+2\nu_1\|\Lambda^\alpha \zeta_t^{(1)}\|^2
+2\nu_2\|\Lambda^\beta \zeta_t^{(2)}\|^2\\
&~~~+C_2(\|\Lambda^\alpha V^{(1)}_{t}\|^{2}
+\|\Lambda^\beta V^{(2)}_{t}\|^{2})\|\zeta_t\|_1
+\frac{\nu_1}{2}\|\Lambda^\alpha V^{(1)}_{t}\|^{2}
+\frac{\nu_2}{2}\|\Lambda^\beta V^{(2)}_{t}\|^{2}
+C_3\|\zeta_t\|^2\|\zeta_t\|_1^2\\
&\leq-(1-C_2\|\zeta_t\|_1)\| (\nu_1\|\Lambda^\alpha V^{(1)}_{t}\|^{2}+\nu_2\|\Lambda^\beta V^{(2)}_{t}\|^{2})
+2\nu_1\|\Lambda^\alpha \zeta_t^{(1)}\|^2
+2\nu_2\|\Lambda^\beta \zeta_t^{(2)}\|^2
+C_3\|\zeta_t\|^2\|\zeta_t\|_1^2,
\end{aligned}
\end{equation*}
where $C_1,C_2,C_3>0$ are constants depending on $\nu_1,\nu_2$. 

For any $T,\delta>0$, we define
\begin{equation*}
\begin{array}{rcl}
\Omega'(\delta,T)&=&\left\{f=(f_{s})_{s\in[0,T]}\in D([0,T];H):\sup\limits_{s\in[0,T]}\left(\max\{\|\Lambda^\alpha f_{s}\|,\|\Lambda^\beta f_{s}\|\}\right)\leq\min\{\delta,\frac{3}{2C_2}\}\right\}.
\end{array}
\end{equation*}
For any $\zeta\in \Omega'$, we can find a constant $C_4$ depending solely on the domain for which
\begin{equation*}
\|V(t)\|^2\leq\|V(0)\|^2e^{-\frac{\nu}{2}t}
+C_4\left[
\min\left(\delta,\frac{3}{2C_2}\right)^4
+\min\left(\delta,\frac{3}{2C_2}\right)^2\right].
\end{equation*}
With $\mathcal{C}$ and $\gamma$  specified in Proposition \ref{MHDirre}, whenever $\|U_0\|\leq \mathcal{C}$ holds, there exist $T>0$ and $\delta>0$ satisfying
$$
\|V_T\|\leq\frac{\gamma}{2}~\text{and}~\delta\leq\frac{\gamma}{2}.
$$
In summary, we arrive at the conclusion that 
\begin{equation}\label{MHD5.5}
\|U_0\|\leq \mathcal{C}\mathrm{~~and~~}\zeta\in\Omega^{\prime}(\delta,T)
\Rightarrow\|U_T\|\leq\|V_T\|+\|\zeta_T\|\leq\gamma.
\end{equation}
Note that for any $T>0$, $\varepsilon>0$ and non-zero real numbers $\alpha_k^l,k\in\mathcal{Z}_0,l\in\{0,1\}$, there holds the following result 
\begin{equation*}
\mathbb{P}(\sup_{t\in[0,T]}\|\sum_{k\in\mathcal{Z}_0}
\alpha_k^lW_{S_t}^k \sigma_k^l\|<\varepsilon)
\geq p_0>0,
\end{equation*}
where $p_0=p_0(T,\varepsilon,\{\alpha_k^l\}_{k\in\mathcal{Z}_0,
l\in\{0,1\}})$.
Inserting this estimate into \eqref{MHD5.5} yields the required conclusion, thereby completing the proof.
\end{proof}

The following result demonstrates the uniqueness of statistically invariant states for \eqref{MHD1.1}.

\begin{proof}[Proof of Theorem \ref{MHDergodicity}]
We first prove the existence of invariant measures. By Lemma \ref{MHDlemma2.1}, it holds that
\begin{equation*}
\nu \mathbb{E}\left[\int_{0}^{t}\|U_{s}\|_{1}^{2}\mathrm{d}s\right]
\leq\|U_{0}\|^{2}+C_{2}t,\quad\forall t\geq0.
\end{equation*}
By adapting the methodology developed in \cite{Peng-2020} and applying the Krylov-Bogoliubov averaging technique, we establish the existence of invariant measure for the system.

To prove the uniqueness, suppose by contradiction that there exist two distinct invariant probability measures $\mu_1$ and $\mu_2$ associated with the Markov semigroup $\{P_t\}_{t\geq0}$. Through an application of Proposition \ref{MHDpropo1.4} combined with the geometric ergodicity arguments from \cite{Kapica-2012}, we derive the following contradiction
\begin{equation}\label{MHDsupp}
\mathrm{Supp~}\mu_1\cap\mathrm{Supp~}\mu_2=\emptyset.
\end{equation}

On the other hand, 
Lemma \ref{MHDlemma2.1} establishes the uniform moment bound
\begin{equation*}
\int_H\|U\|^2\mu(\mathrm{d}U)\leq C.
\end{equation*}
for every invariant measure $\mu$.
Furthermore, by adapting the spectral gap argument from \cite[Corollary 4.2]{MH-2006} and combining it with the irreducibility property in Proposition \ref{MHDirre}, we deduce that the origin must belong to the support of every invariant measure, i.e.,
\begin{equation}
0 \in \mathrm{Supp~}\mu.
\end{equation}
This conclusion directly contradicts the support condition specified in \eqref{MHDsupp}. The resulting contradiction establishes the uniqueness of the invariant measure. As an immediate consequence, we obtain the ergodicity of the system.
%
\end{proof}

%

\subsection*{Conflict of interest}
This work does not have any conflicts of interest.

\subsection*{Availability of date and materials}
Not applicable.

\section*{Acknowledgment}
The authors would like to thank Professor Xuhui Peng from Hunan Normal University, China, Changsha for his helpful discussion to this manuscript.
\newcounter{cankan}


\begin{thebibliography}{100}
\setlength{\itemsep}{-2mm}

\bibitem{Barbu-2007} Barbu, V., Da Prato, G.: Existence and ergodicity for the two-dimensional stochastic Magnetohydrodynamics equations. Appl. Math. Optim. \text{56}, 145-168 (2007)
    
\bibitem{Cababbes} Cababbes, H.: Theoretical Magne to Fluid Dynamics. Academic Presee, New York, (1970)

\bibitem{Constantin} Constantin, P., Doering, R. Charles: Infinite prandtl number convection. J. Stat. Phys. \text{94}(1), 159–172 (1999)
    
\bibitem{Cowling} Cowling, T., Phil, D.: Theoretical Magnetofluiddynamics. The institute of Physics, (1976)

\bibitem{prato-1996} Da Prato, G., Zabczyk, J.: Ergodicity for Infinite Dimensional Systems, volume 229 of London Mathematical Society Lecture Note Series. Cambridge University Press, Cambridge, (1996)

\bibitem{EM01} E, W., Mattingly, J.C.: Ergodicity for the Navier-Stokes equation with degenerate random forcing: finite-dimensional approximation. Comm. Pure Appl. Math. \text{54}(11), 1386–1402 (2001)


\bibitem{Foias-2001} Foias, C., Manley, O., Rosa, R., Temam, R.: Navier–Stokes Equations and Turbulence, Encyclopedia Math. Appl., vol. 83, Cambridge University Press, Cambridge, (2001)

\bibitem{Foias-1967} Foias, C., Prodi, G.: Sur le comportement global des solutions non-stationnaires des équations de Navier–Stokes en dimension 2. Rend. Semin. Mat. Univ. Padova. \text{39}, 1–34 (1967)

\bibitem{Foldes-2015} F\"{o}ldes, J., Glatt-Holtz, N., Richards, G., Thomann, E.: Ergodic and mixing properties of the Boussinesq equations with a degenerate random forcing. J. Funct. Anal. \text{269}(8), 2427–2504 (2015)
    
\bibitem{Glatt-2008} Glatt-Holtz, N., Ziane, M.: Strong pathwise solutions of the stochastic Navier-Stokes system. Adv. Differential Equations \text{13}(5-6), 567-600 (2008)

\bibitem{MH-2006} Harier, M., Mattingly, J.C.: Ergodicity of the 2D Navier-Stokes equations with degenerate stochastic forcing. Ann. Math. \text{164}, 993–1032 (2006)

\bibitem{MH-2008} Hairer, M., Mattingly, J.C.: Spectral gaps in Wasserstein distances and the 2D stochastic Navier-Stokes equations. Ann. Appl. Probab. \text{36}(6), 2050-2091 (2008)

\bibitem{MH-2011} Hairer, M., Mattingly, J.C.: A theory of hypoellipticity and unique ergodicity for semilinear stochastic PDEs. Electron. J. Probab. \text{16}(23), 658–738 (2011)
    
\bibitem{Hong-2021}  Hong, W., Li, S., Liu, W.: Well-posedness and exponential mixing for stochastic magneto-hydrodynamic equations with fractional dissipations. Front. Math. China \text{16}(2), 2021 425-457.
    
\bibitem{Huang-2025} Huang, J., Peng, X., Wang, X., Zhang, J.: Ergodicity for stochastic 2D Boussinesq equations with a highly degenerate pure jump Lévy noise. arXiv:2503.18045

    
\bibitem{Huang-2016} Huang, J., Shen, T.: Well-posedness and dynamics of the stochastic fractional magneto-hydrodynamic equations. Nonlinear Anal.: Theory, Methods Appl. \text{133}, 102-133 (2016)

\bibitem{Huang-2022} Huang, J., Zheng, Y., Shen, T., Guo, C.: Asymptotic properties of the 2D stochastic fractional Boussinesq equations driven by degenerate noise. J. Differential Equations \text{310}, 362-403 (2022)

\bibitem{Imkeller-2006} Imkeller, P., Pavlyukevich, I.: First exit times of SDEs driven by stable L\'{e}vy processes. Stoch. Process. Appl. \text{116}(4), 611–642 (2006)

\bibitem{Kapica-2012} Kapica, R., Szarek, T., Sleczka, M.: On a unique ergodicity of some Markov processes. Potential Anal. \text{36}(4), 589–606 (2012)
 
\bibitem{Komorowski-2010} Komorowski, T., Peszat, S., Szarek, T.: On ergodicity of some Markov processes. Ann. Probab. \text{38}(4), 1401–1443 (2010)

\bibitem{Kuchler-2013} K\"{u}chler, U., Tappe, S.: Tempered stable distributions and processes. Stoch. Process. Appl. \text{123}(12), 4256–4293 (2013)
    
\bibitem{Luo-2022} Luo, D.: Regularization by transport noises for 3D MHD equations. Sci China Math, 65 (2022)

\bibitem{Peng-Huang-2020} Peng, X., Huang, J., Zheng, Y.: Exponential mixing for the fractional magneto-hydrodynamic equations with degenerate stochastic forcing. Commun. Pure Appl. Anal. \text{19}(9), 4479-4506 (2020)

\bibitem{Peng-2020} Peng, X., Huang, J., Zhang, R.: Ergodicity and exponential mixing of the real Ginzburg-Landau equation with a degenerate noise. J. Differential Equations \text{269}(4), 3686-3720 (2020)

\bibitem{Peng-2024}  Peng, X., Zhai, J., Zhang, T.: Ergodicity for 2D Navier-Stokes equations with a degenerate pure jump noise. arXiv:2405.00414
    
\bibitem{Peng-2025} Peng, X., Su, H.:  Ergodicity of the viscous scalar conservation laws with a degenerate noise. arXiv:2503.15828
    
\bibitem{Sermange-1983} Sermange, M., Temam, R.: Some mathematical questions related to the MHD equations. Comm. Pure Appl. Math. \text{36}(5), 635-664 (1983)
    

\bibitem{Zhang-2014}  Zhang, X.: Densities for SDEs driven by degenerate $\alpha$-stable processes. Ann. Probab. \text{42}(5), 1885–1910 (2014)




\end{thebibliography}
\end{document}